\documentclass[11pt]{amsart}

\usepackage{fancyhdr}
\usepackage{layout}
\usepackage{caption} 
\usepackage{kotex}
\usepackage{amsthm,amsxtra}
\usepackage{a4wide}
\usepackage{amssymb}
\usepackage{amsmath}
\usepackage{amsfonts}
\usepackage[pdftex]{color,graphicx}
\usepackage[dvipsnames]{xcolor}
\usepackage{anyfontsize}

\usepackage{latexsym}
\usepackage{mathrsfs}
\usepackage{enumerate}
\usepackage{verbatim}
\usepackage{lipsum}

\usepackage{here}
\usepackage{amscd,amstext}
\usepackage{graphicx}
\usepackage{epstopdf}

\usepackage[colorlinks]{hyperref}
\hypersetup{
linkcolor=blue,      
citecolor=green      
}

\usepackage{tikz}

\newtheorem{thm}{Theorem}[section]
\newtheorem{lem}[thm]{Lemma}
\newtheorem{cor}[thm]{Corollary}
\newtheorem{prop}[thm]{Proposition}

\newtheorem{defn}[thm]{Definition}
\newtheorem{remark}[thm]{Remark}
\newtheorem{exam}[thm]{Example}

\numberwithin{equation}{section}
\newtheorem*{thmnn}{Theorem}

\newenvironment{proofof(1)}{\paragraph{\textit{Proof of (1)}}}{\hfill$\square$}

\newcommand{\boldotimes}{{\fontsize{5}{7}\selectfont {$\pmb\otimes$}}}
\newcommand{\B}{\mathcal{B}}

\newcommand{\G}{\mathcal{G}}

\newcommand{\metricd}{\mathbf{d}}
\newcommand{\radiusr}{\mathbf{r}}
\newcommand{\bu}{\mathbf{u}}
\newcommand{\initN}{N_0}

\newcommand{\maxNN}{N_1}
\newcommand{\maxNNn}{N\!_1}
\newcommand{\capitalK}{K}
\newcommand{\cT}{T}
\newcommand{\cX}{\mathcal{X}}
\newcommand{\cZ}{\mathcal{Z}}
\newcommand{\xNone}{x\!_{_{N\!_1}}\!}
\newcommand{\tildexNone}{\tilde x\!_{_{N\!_1}}\!}

\title{Quasi-Sturmian colorings on regular trees}

\author{Dong Han Kim}
\address{Department of Mathematics Education, Dongguk University–Seoul, 30 Pildong-ro 1-gil, Junggu, Seoul, 04620}
\email{kim2010@dongguk.edu}

\author{Seul Bee Lee}
\address{Department of Mathematical Sciences, Seoul National University, Kwanak-ro 1, Kwanak-gu,
Seoul, 08826}
\email{seulbee.lee@snu.ac.kr}

\author{Seonhee Lim}
\address{Department of Mathematical Sciences, Seoul National University, Kwanak-ro 1, Kwanak-gu,
Seoul, 08826}
\email{seonhee.lim@gmail.com}

\author{Deokwon Sim}
\address{Department of Mathematical Sciences, Seoul National University, Kwanak-ro 1, Kwanak-gu,
Seoul, 08826}
\email{deokwon.sim@snu.ac.kr }

\date{\today}

\begin{document}

\maketitle

\begin{abstract}
Quasi-Sturmian words, which are infinite words with factor complexity eventually $n+c$ share many properties with Sturmian words.
In this paper, we study the quasi-Sturmian colorings on regular trees. There are two different types, bounded and unbounded, of quasi-Sturmian colorings.
We obtain an induction algorithm similar to Sturmian colorings.
We distinguish them by the recurrence function.
\end{abstract}

\section{Introduction}

Factor complexity $p_\bu(n)$ of an infinite word $\bu$ is the number of distinct subwords of length $n$ appearing in $\bu$.
Factor complexity is a classical invariant which measures the disorder of words. 
Hedlund and Morse showed that a word is eventually periodic if and only if its factor complexity is bounded (see \cite{MH}).
Sturmian words are infinite words which have unbounded minimal factor complexity, i.e. $p_\bu(n)=n+1$. Sturmian words are cutting sequences of irrational rotations and enjoy many striking properties \cite{Ra}, \cite{Ca}, \cite{Lo}. An induction algorithm is developed using the Rauzy graph, which is a graph whose vertices are the subwords of length $n$ in $\bu$.

An infinite word $\bu$ is \emph{quasi-Sturmian} if there are integers $c$ and $\initN$ such that $p_\bu(n)=n+c$ for $n\ge \initN$.
Cassaigne showed that a quasi-Sturmian word is an image of a Sturmian word by a non-periodic morphism (see \cite{Ca}).

Factor complexity and Sturmian words have been generalized to Sturmian colorings of trees by the first and third authors \cite{KL}. By a \emph{coloring} of a regular tree $\cT$, we mean a vertex coloring with finite alphabet, i.e. a surjective map $\phi:V\cT\rightarrow\mathcal{A}$ from the vertex set $V\cT$ to the set $\mathcal{A}$ of {\it{alphabet}}, such that $|\mathcal{A}|<\infty$. For Sturmian colorings, we developed an induction algorithm using the graph of colored balls of radius $n$, which are analogs of subwords of length $n$ in a Sturmian word \cite{KL2}.


For subtrees $\cT_1$ and $\cT_2$ of $\cT$, we define a \emph{color-preserving homomorphism} $f:\cT_1\rightarrow \cT_2$ of a coloring $\phi$ as a graph homomorphism such that $\phi(v)=\phi(f(v))$ for all $v\in V\cT_1$.
Let $\mathbf d$ be the metric on $\cT$ giving length $1$ on each edge.

\begin{defn}
The \emph{$n$-ball $\mathcal{B}_n(u)$ of center $u$} is the closed $\mathbf{d}$-ball of radius $n$ and center $u$. Two $n$-balls $\mathcal{B}_{n}(u)$ and $\mathcal{B}_{n}(v)$ are \emph{equivalent} if there is a color-preserving isometry $f:\mathcal{B}_{n}(u)\rightarrow\mathcal{B}_{n}(v)$.  
We denote the equivalence class of $\mathcal{B}_{n}(u)$ by $[\mathcal{B}_{n}(u)]$ and call it a \emph{colored $n$-ball}.
Let $\mathbb{B}_{\phi}(n)$ be the set of colored $n$-balls of $\phi$.

The \emph{factor complexity} $b_\phi (n) = |\mathbb{B}_\phi(n)|$ of a coloring $\phi$ is the number of colored $n$-balls in $(\cT, \phi)$.
\end{defn}
By convention, we denote by $\B_{-1}(u)$ the empty ball and let $b_\phi(-1) = 1$.
Clearly, $b_\phi(0)=|\mathcal{A}|$.
Factor complexity is either a bounded function or a strictly increasing function (see Theorem 2.7 in \cite{KL}).

For a given coloring $\phi$, let $\Gamma$ be the group of color-preserving isometries of $\cT$.
 {
The quotient $X=\Gamma\backslash \cT$ has a structure of an \emph{edge-indexed graph}, which is a graph equipped with an index map $i:EX\rightarrow \mathbb{N}$ defined as follows: 
}
Let $e\in EX$ be an oriented edge with the initial vertex $x\in VX$ and the terminal vertex $y\in VX$.
Let $\tilde{x}$ be a lift of $x$ in $\cT$.
The index $i(e)$ is the number of lifts of $y$ among the neighboring vertices of $\tilde{x}$.
We will sometimes denote $e$ by $[x,y]$ and denote $i(e)$ by $i(x,y)$.
We call $\cX=(X,i)$ the \emph{quotient (edge-indexed) graph of $(\cT,\phi)$}. 
Let $\pi:\cT\rightarrow X$ be the covering map.
There is a coloring $\phi_0$ of $X$ such that $\phi=\phi_0\circ\pi$. 

We say that a coloring is \emph{periodic} if its quotient graph is a finite graph.
A coloring $\phi$ is periodic if and only if the factor complexity $b_\phi(n)$ is bounded (see \cite{KL}).
A \emph{Sturmian coloring} $(\cT,\phi)$ is a coloring with unbounded minimal factor complexity  $b_\phi(n)=n+1$.

\begin{defn}\label{Def:Quasi-Sturmian}
We say that a coloring is \emph{quasi-Sturmian} if there exists a pair of integers $c$ and $N_0$  such that $b(n)=n+c$ for $n\geq N_0$, i.e. 
\begin{equation}\label{uniquespecial} b(n+1)-b(n)=1 \textrm{ for each } n\geq N_0.\end{equation}
\end{defn}
We assume that $\initN$ is the minimal integer satistying (\ref{uniquespecial}).
If $b(n+1)>b(n)$, there are at least two distinct $n$-balls $\B_{n}(u)$ and $\B_{n}(v)$ such that $[\B_{n}(u)]=[\B_{n}(v)]$ but $[\B_{n+1}(u)]\not=[\B_{n+1}(v)]$.
We call such a colored $n$-ball $[\B_{n}(u)]$ \emph{special}.
A quasi-Sturmian coloring has a unique special $n$-ball for all $n \geq \initN$ which we denote by $S_n$. 

In \cite{KL}, it was shown that a quotient graph of a Sturmian coloring is either a geodesic ray or an infinite geodesic with loops possibly attached to each vertex. 
The \emph{type set} $\Lambda_u$ of a vertex $u\in V\cT$ is the set of non-negative integers $n$ for which $[\B_n(u)]$ is special. 
A vertex $u$ is said to be of \emph{bounded type} if $\Lambda_u$ is a finite set. 
 {For a vertex $u$ of bounded type, the maximal type $\tau(u)$ of $u$ is the maximum of elements in $\Lambda_u$.}

We say that a coloring $\phi$ is of \emph{bounded type} if each vertex (or equivalently a vertex) of $(\cT, \phi)$ is of bounded type.
Otherwise, we say that a coloring $\phi$ is of \emph{unbounded type}.
For a coloring of bounded type, we define the subgraph $G$ of $X$ as the graph consisting of the vertices whose lifts are of maximal type less than or equal to $\maxNN$ (see the equation \eqref{Eq:N_1} for the definition).

We first characterize the quotient graph of a quasi-Sturmian coloring.

\begin{thm}[Quotient graphs of quasi-Sturmian colorings]\label{Thm:Main1}
If $\phi$ is a quasi-Sturmian coloring, then the quotient graph is one of the following graphs.
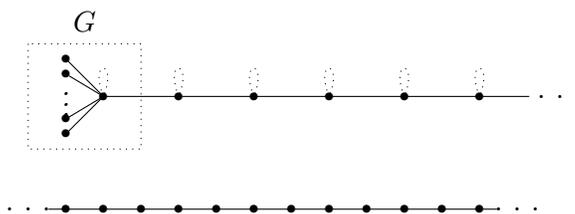
\begin{figure}[h]
\begin{center}\begin{tikzpicture}[every loop/.style={}]
  \tikzstyle{every node}=[inner sep=-1pt]
  \node(-1) at (-3.5, 1  ) {\tiny$\bullet$};
  \node(0)  at (-3.5, 0  ) {\tiny$\bullet$};
  \node(1)  at (-3   ,0.5) {\tiny$\bullet$};
  \node(2)  at (-2   ,0.5) {\tiny$\bullet$};
  \node(3)  at (-1   ,0.5) {\tiny$\bullet$};
  \node(4)  at ( 0   ,0.5) {\tiny$\bullet$};
  \node(5)  at ( 1   ,0.5) {\tiny$\bullet$};
  \node(6)  at ( 2   ,0.5) {\tiny$\bullet$};
  \node(7)  at ( 3   ,0.5) {$~.~.~.$};

  \node(8) at (-3.5,0.8){\tiny$\bullet$};
  \node(9) at (-3.5,0.5){$\vdots$};
  \node(10) at (-3.5,0.2){\tiny$\bullet$};

\draw[dotted] (-3,0.5) .. controls (-3.2,1) and (-2.8,1) .. (-3,0.5);
\draw[dotted] (-2,0.5) .. controls (-2.2,1) and (-1.8,1) .. (-2,0.5);
\draw[dotted] (-1,0.5) .. controls (-1.2,1) and (-0.8,1) .. (-1,0.5);
\draw[dotted] (0,0.5) .. controls (-0.2,1) and (0.2,1) .. (0,0.5);
\draw[dotted] (1,0.5) .. controls (0.8,1) and (1.2,1) .. (1,0.5);
\draw[dotted] (2,0.5) .. controls (1.8,1) and (2.2,1) .. (2,0.5);

\tikzstyle{every loop}=   [-, shorten >=.5pt]

  \path[-] 
	(-1) edge (1)
	(0)  edge (1)
	(1)  edge (2)	 
	(2)  edge (3)	
	(3)  edge (4)
	(4)  edge (5)	
	(5)  edge (6)
	(6)  edge (7)
	(8)  edge (1)
	(10) edge (1);

\draw[dotted] (-4,1.2) rectangle (-2.5,-0.2);
\node at (-3.25, 1.5) {$G$};

  \tikzstyle{every node}=[inner sep=-1pt]

  \node(-116)  at (-4   ,-1) {$.~.~.$};
  \node(-115)  at (-3.5   ,-1){\tiny$\bullet$};
  \node(-114)  at (-3   ,-1) {\tiny$\bullet$};
  \node(-113)  at (-2.5   ,-1) {\tiny$\bullet$};
  \node(-112)  at (-2   ,-1) {\tiny$\bullet$};
  \node(-111)  at (-1.5   ,-1) {\tiny$\bullet$};
  \node(110)  at ( -1   ,-1) {\tiny$\bullet$};
  \node(111)  at ( -0.5   ,-1) {\tiny$\bullet$};
  \node(112)  at ( 0   ,-1) {\tiny$\bullet$};
  \node(113)  at ( 0.5   ,-1) {\tiny$\bullet$};
  \node(114)  at (1   ,-1) {\tiny$\bullet$};
  \node(115)  at (1.5   ,-1) {\tiny$\bullet$};
  \node(116)  at (2   ,-1) {\tiny$\bullet$};
  \node(117)  at (2.5   ,-1) {$.~.~.$};

\tikzstyle{every loop}=   [-, shorten >=.5pt]

  \path[-] 
   (-116) edge (-115)
   (-115) edge (-114)
   (-114) edge (-113)
   (-113) edge (-112)
   (-112) edge (-111)
   (-111) edge (110)
   (110) edge (111)
	(111)  edge (112)	 
	(112)  edge (113)	
	(113)  edge (114)
	(114)  edge (115)	
	(115)  edge (116)
	(116)  edge (117);
\end{tikzpicture}\end{center}
\end{figure}

More precisely, the quotient graph of a coloring of bounded type is the first graph, where as the quotient graph of a coloring of unbounded type is a geodesic ray or a biinfinite geodesic.
\end{thm}

\begin{figure}[h]
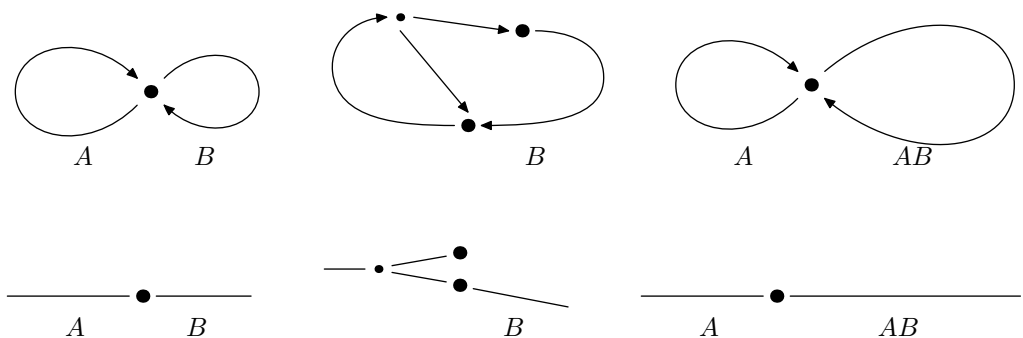

\includegraphics{quasisturmian_graph-2.mps} \qquad
\includegraphics{quasisturmian_graph-3.mps} \qquad 
\includegraphics{quasisturmian_graph-4.mps}

\vspace{1cm}

\includegraphics{quasisturmian_graph-12.mps}  \qquad
\includegraphics{quasisturmian_graph-13.mps} \qquad
\includegraphics{quasisturmian_graph-14.mps} 
\caption{The evolution of Rauzy graphs of a quasi-Sturmian word (above) and The evolution of $\G_n$ of a quasi-Sturmian coloring on a tree (below)}\label{Rauzygraph}
\end{figure}

Our second main theorem is the induction algorithm of quasi-Sturmian colorings.
As was mentioned in the beginning, the Rauzy graph (or the factor graph) of an infinite word $\bu$ is an oriented graph whose vertices are the factors (or subwords) of length $n$. 
For two factors $\mathbf{v}$ and $\mathbf{w}$ of length $n$, there is an edge from $\mathbf{v}$ to $\mathbf{w}$ if there are letters $a$ and $b$ such that $\mathbf{v}a=b\mathbf{w}$.
There is a continued fraction algorithm for Sturmian words in terms of Rauzy graphs (see \cite{Ra} for details).  
The evolution of Rauzy graph is shown in Figure~\ref{Rauzygraph}.

Analogous factor graph for colorings of trees is defined in \cite{KL2}: the factor graph $\G_n$ is defined as the graph whose vertices are the colored $n$-balls. Its edges are pairs of colored $n$-balls appearing in $(\cT, \phi)$ of distance $1$, i.e. $(D_n, E_n)$ with vertices $v,w \in V\cT$ such that $[\B_n(v)]=D_n$, $[\B_n(w)]=E_n$ with $\metricd (v,w)=1$.
An induction algorithm of $\G_n$ for Sturmian colorings is developed in \cite{KL2} (see Figure~\ref{Rauzygraph}).

The following theorem shows how the factor graphs $\G_n$ of quasi-Sturmian colorings evolve as $n$ goes to infinity.

\begin{thm}[Induction algorithm]\label{Thm:Main2}
For an acyclic quasi-Sturmian coloring, the factor graph $\G_n$ falls into one of the three cases:

\begin{enumerate}
\item[(I)] $S_n \neq C_n$ and one of $S_n$, $C_n$ is $A_n$ or $B_n$:
 $\G_n$ is a linear graph, i.e., all vertices have degree less than or equals to 2.

\item[(II)] $S_n, A_n, B_n, C_n$ are all distinct: $\G_n$ is a tripod and $S_n$ is the unique vertex of degree 3 in $\G_n$.

\item[(III)] $S_n, A_n, B_n$ are distinct and $C_n = S_n$: The special ball $S_n$ is a vertex of degree 2 in $\G_n$.
\end{enumerate}
Suppose that $\G_n$ corresponds to (I) and $m$ is the number of vertices of $\G_n$ which are connected to $S_n$ through $C_n$ in $\G_n$.
Then $\G_{n+i}$ corresponds to (II) for $0 < i \le m$ and either $\G_{n+m}$ corresponds to (I) again or  $\G_{n+m}$ corresponds to (III)  and  $\G_{n+m+1}$ corresponds to (I):
%
%
$$\text{(I)} \to \text{(II)} \to \cdots \to \text{(II)} \to \text{(I)} \quad \text{ or }  \quad  \text{(I)} \to \text{(II)} \to \dots \to \text{(II)} \to \text{(III)} \to \text{(I)}. 
$$
\end{thm}

In Section~\ref{Sec_bounded}, we further investigate the quotient graph $\cX$ of a quasi-Sturmian coloring of bounded type.
We will show that we can obtain $\cX$ by attaching a finite edge indexed colored graph to a geodesic ray whose lift is a part of a periodic coloring.
More precisely, for a subgraph $Y$ of $\cT$, a \emph{periodic extension to $\cT$} of a coloring of $Y$ is a periodic coloring on $\cT$ which coincides with the coloring on $Y$.

\begin{thm}[Quotient graphs of colorings of bounded type]\label{Thm:Main3}
Let $\cX=(X,i)$ be the quotient graph of a coloring $(\cT,\phi)$.
The following statements are equivalent.
\begin{enumerate}
\item The coloring $\phi$ is a quasi-Sturmian coloring of bounded type.
\item There is a finite connected subgraph $G$ of the quotient graph $X$ such that $X-G$ is a connected infinite ray and any connected component of $\cT-\widetilde{G}$ has a periodic extension to $\cT$ where $\widetilde{G}$ is the union of lifts of $G$.
\end{enumerate}
\end{thm}

In the last part of the article, we obtain some bounds for the recurrence functions of quasi-Sturmian colorings. 
There are three versions of recurrence functions for words (see \cite{AS}, \cite{Ca}).
In this paper, we mainly consider one of the recurrence functions $R''$ which is defined by Cassaigne:
for infinite words, let 
$R''_\bu (n)$ be the length of the smallest factor of $\mathbf u$ containing all factors of length $n$.
It is clear that $R''_\bu(n)\ge p_\bu(n)+n-1$.
If $R''_\bu(n)=p_\bu(n)+n-1$ for each $n\ge 0$, then we say that $\bu$ has \emph{grouped factors}.
Cassaigne proved that quasi-Sturmian words have \emph{ultimately grouped factors}, i.e., there is $n_0$ such that $R''_\bu (n)=p_\bu+n-1$ for all $n\ge n_0$ (see \cite{Ca}). 

We define the recurrence function $R''(n)$ for colorings of trees analogously.
If a quasi-Sturmian coloring $\phi$ is of unbounded type, the factor graph $\G_n$ is of type (I) on $n=n_k$ in Theorem~\ref{Thm:Main2} (see Theorem~\ref {Thm:AnIffConditionOfTheQuotientGraphOfBoundedType}).
Denote by  {$\cZ=(Z, i_Z)$} the finite quotient graph of $\cT-\widetilde{G}$ with respect to the coloring $\phi$ in Theorem ~\ref{Thm:Main3}.

For a finite graph $G$ and a vertex $x$ of $G$, let us denote by 
$$
\mathrm{\radiusr}(x,G)=\max\{\metricd(x,y):y\in VG\}.
$$


\begin{thm}[Bounds of $R''(n)$]\label{Thm:Main4}
Let $(\cT,\phi)$ be a quasi-Sturmian coloring and $\cX=(X,i)$ be the quotient graph of $(\cT,\phi)$.

\begin{enumerate}
\item Let $\phi$ be of unbounded type.
Then, we have $$R''(n) = n+\Bigl\lfloor\frac{b(n_k)}{2}\Bigr\rfloor \quad \mathrm{for} \;\; n_{k-1}<n \le n_k.$$.

\item Let $\phi$ be of bounded type. Let $\xNone$ be the vertex of maximal type $\maxNN$. 

\begin{enumerate}
\item If $Z$ is acyclic, then we have $$R''(n) = n+\Bigl\lfloor\frac{1}{2}(b(n_k)-|G|+\radiusr(x\!_{_{\maxNNn}}\!,G)+1)\Bigr\rfloor \quad \mathrm{for} \;\; n_{k-1}<n \le n_k.$$

\item If $Z$ is cyclic, then we have $$R''(n) = n+\Bigl\lfloor\frac{1}{2}(b(n)-|G|+\radiusr(x\!_{_{\maxNNn}}\!,G)+1)\Bigr\rfloor \quad \mathrm{for} \;\; n \ge \maxNN.$$ 
\end{enumerate}
\end{enumerate}

\end{thm}


The article is organized as follows. In Section~\ref{sec_Quotient graphs of quasi-Sturmian colorings}, we characterize quotient graphs of quasi-Sturmian colorings. Theorem~\ref{Thm:Main1} is proved in
Proposition~\ref{Prop:TheQuotientGraphOfBoundedType} and Proposition~\ref{Prop:ForUnboundedType,TheNeighboringVerticesHaveAtMost3TypeSets}.
In Section~\ref{Sec_CF}, we study the factor graphs of quasi-Sturmian colorings.
We prove Theorem~\ref{Thm:Main2} in Proposition~\ref{graph_develop}.
We prove Theorem~\ref{Thm:Main3} in Theorem~\ref {Thm:AnIffConditionOfTheQuotientGraphOfBoundedType}.
In Section~\ref{Sec_recurrence}, we investigate the recurrence function of quasi-Sturmian colorings.
\section{Quotient graphs of quasi-Sturmian colorings}\label{sec_Quotient graphs of quasi-Sturmian colorings}

In this section, we characterize the quotient graphs of quasi-Sturmian colorings.
The quotient graph of 
 a quasi-Sturmian coloring of bounded type is a union of a finite graph and a geodesic ray.
For a quasi-Sturmian coloring of unbounded type, the quotient graph is a geodesic ray or an infinite geodesic.

Recall that for a vertex $u$ of bounded type, the maximal type $\tau(u)$ of $u$ is the maximum of elements in $\Lambda_u$. If a vertex of a coloring of a tree is of bounded type, then every vertex is of bounded type (see Lemma 2.15 in \cite{KL}).

\subsection{Quotient graphs of quasi-Sturmian colorings of bounded type}
For $u\in V\cT$, $\tau(u)\leq m$ if and only if $[\B_{m+1}(u)]=[\B_{m+1}(v)]$ implies that $u$ and $v$ are in the same class. 
If two vertices $u$ and $v$ are in the same class, then $u$ and $v$ have the same maximal type.
Kim and Lim proved that the converse is also true in the case of a Sturmian coloring (see Proposition 3.2 in \cite{KL}). We observe that the same proof holds in quasi-Sturmian colorings as long as $b(n+1)-b(n)=1$. We provide the proof for completeness.

We say that two vertices $u,v$ are \emph{in the same class} if there is a color-preserving isometry of $\cT$ such that $f(u) = v$. Note that two vertices are in the same class if the $n$-balls $\mathcal{B}_n(u)$, $\mathcal{B}_n(v)$ are equivalent for every $n$.
\begin{lem}
Suppose that $b(n)$ is a strictly increasing function. 
If $b(n+1)-b(n)=1$ and two vertices $u$ and $v$ have maximal type $n$, then $u$ and $v$ are in the same class.
\end{lem}

\begin{proof}
Suppose that $b(n+1)-b(n)=1$ and there exist two vertices $u$ and $v$ not in the same class such that $\tau(u)=\tau(v)=n$. Since the alphabet $\mathcal{A}$ is finite, there is a number $N$ such that $\B_{N}(w)$ contains a special $n$-ball for each $w\in V\cT$ (see Lemma 2.16 in \cite{KL}).

Fix a vertex $w$ and let $z$ be the center of a special $n$-ball contained in $\B_{N}(w)$.
Since the special $n$-ball is unique, either $[\B_{n+1}(z)]=[\B_{n+1}(u)]$ or $[\B_{n+1}(z)]=[\B_{n+1}(v)]$, thus $z$ is in the same class of $u$ or $v$.
Since $w\in \B_{N}(z)$, the tree $\cT$ is covered by $N$-balls whose centers are in the same class of $u$ or $v$. 
Thus, the maximal types of vertices of $\cT$ is bounded by $M=\max\{\tau(p):p\in\B_{N}(u)\cup\B_{N}(v)\}$.
It contradicts that $b(n)$ is strictly increasing.
\end{proof}

\begin{cor}\label{Cor:TheSameMaximalTypeImpliesTheSameClass}
Let $(\cT,\phi)$ be a quasi-Sturmian coloring of bounded type with factor complexity $b(n)=n+c$ for $n\ge \initN$.
If two vertices u and v of $(\cT,\phi)$ have the same maximal type greater than or equal to $\initN$, then u and v are in the same class.
\end{cor}

\begin{lem}\label{Lem:TheDiffereceOfMaximalTypesBetweenAdjacentVerticesIsLessThan2} 
If a vertex $u$ of a quasi-Sturmian coloring $(\cT,\phi)$ is of maximal type $m$, then the following hold.
\begin{enumerate}
\item\label{Lem:AdjVer(1)} If $m\ge \initN$, its neighboring vertices are of maximal type $m-1$, $m$, $m+1$.\\
 If $m= \initN-1$, its neighboring vertices are of maximal type at most $\initN$.\\
 If $m\le \initN-2$, its neighboring vertices are of maximal type at most $\initN-1$.
\item\label{Lem:AdjVer(2)} If $m\ge \initN$, one of its neighboring vertices is of maximal type $m+1$.
\item\label{Lem:AdjVer(3)} If $m\ge \initN$ is not minimum among maximal types of vertices, one of its neighboring vertices is of maximal type $m-1$.
\end{enumerate}
\end{lem}

\begin{proof}
Let $\{u_i\}_{i=1,\cdots,d}$ be the neighboring vertices of $u$.

(1) Let $\tau=\max\{\tau(u_i)\}_{i=1,\cdots,d}$.
Choose $u_k$ such that $\tau(u_k)=\tau$.
There is a vertex $v$ such that $[\B_{\tau}(u_k)]=[\B_{\tau}(v)]$ but $[\B_{\tau+1}(u_k)]\not=[\B_{\tau+1}(v)]$. 
Let $f:\B_{\tau}(u_k) \rightarrow \B_{\tau}(v)$ be a color-preserving isometry.
Let $w=f(u)$.
Suppose that $\tau>m+1$.
Since $\B_{m+1}(u)\subset\B_{\tau}(u_k)$, $[\B_{m+1}(u)]=[\B_{m+1}(w)]$.
Thus, $u$ and $w$ are in the same class.
Since $\mathbf{d}(w,v)=1$, $u_j$ and $v$ are in the same class for some $j$.
We have 
$$
[\B_{\tau}(u_j)]=[\B_{\tau}(v)]=[\B_{\tau}(u_k)]\text{ and }[\B_{\tau+1}(u_j)]=[\B_{\tau+1}(v)]\not=[\B_{\tau+1}(u_k)],
$$
thus $\tau(u_j)\geq \tau$.
By the maximality of $\tau$, $\tau(u_j)=\tau$.
By Corollary~\ref{Cor:TheSameMaximalTypeImpliesTheSameClass}, if $\tau>\initN$, then $u_k$ and $u_j$ are in the same class.
It contradicts $[\B_{\tau+1}(u_k)]\not=[\B_{\tau+1}(u_j)]$.
Hence, $\tau<\initN$.

We conclude that $\tau>m+1$ implies $\tau<\initN$.
If $m\ge \initN-1$, then $\tau\le m+1$.
If $m<\initN-1$, then $\tau\le \initN-1$.
In other words, for $u,v$ such that $\mathbf{d}(u,v)$=1, if $|\tau(u)-\tau(v)|\ge 2$, then $\tau(u),\tau(v)\le \initN-1$.
Thus if $m\ge \initN$, then $\tau(u_i)\ge m-1$.

(2) Let $m\ge \initN$.
Suppose that there is no $u_i$ such that $\tau(u_i)=m+1$.
By (\ref{Lem:AdjVer(1)}), $m-1\le \tau(u_i) \le m$ for each $i$.
If $\tau(u_i)=m-1$, then there is no vertices on $\B_{1}(u_j)$ of maximal type greater than $m$.
Even if $\tau(u_i)=m$, since $u$ and $u_i$ are in the same class by Corollary~\ref{Cor:TheSameMaximalTypeImpliesTheSameClass}, we have the same conclusion.
Thus, there is no vertex on $\B_{2}(u)$ of maximal type greater than $m$.
Inductively, every vertex is of maximal type less than $m+1$.
It contradicts the fact that $b(n)$ is strictly increasing.

(3) We can show it by the similar argument of the proof of (\ref{Lem:AdjVer(2)}).
\end{proof}

For a quasi-Sturmian coloring of bounded type, we define 
\begin{equation}\label{Eq:N_1}
\maxNN = \max\{\initN,\textrm{ }\min\{\tau(x):x \in V\cT\}\}.
\end{equation}
For a coloring of bounded type, we define the subgraph $G$ of $X$ as the graph consisting of the vertices of maximal type less than or equal to $\maxNN$.
The next proposition follows from Corollary~\ref{Cor:TheSameMaximalTypeImpliesTheSameClass} and Lemma~\ref{Lem:TheDiffereceOfMaximalTypesBetweenAdjacentVerticesIsLessThan2}.

\begin{prop}\label{Prop:TheQuotientGraphOfBoundedType}
For the quotient graph $\cX=(X,i)$ of a quasi-Sturmian coloring $\phi$ of bounded type, the quotient  graph $X$ is a union of $G$ and a geodesic ray (see the following figure). 

\begin{figure}[h]
\begin{center}
\begin{tikzpicture}[every loop/.style={}]
  \tikzstyle{every node}=[inner sep=-1pt]
  \node(-1) at (-3.5, 1  ) {$\bullet$};
  \node(0)  at (-3.5, 0  ) {$\bullet$};
  \node(1)  at (-3   ,0.5) {$\bullet$};
  \node(2)  at (-2   ,0.5) {$\bullet$};
  \node(3)  at (-1   ,0.5) {$\bullet$};
  \node(4)  at ( 0   ,0.5) {$\bullet$};
  \node(5)  at ( 1   ,0.5) {$\bullet$};
  \node(6)  at ( 2   ,0.5) {$\bullet$};
  \node(7)  at ( 3   ,0.5) {$~.~.~.$};

  \node(8) at (-3.5,0.8){$\bullet$};
  \node(9) at (-3.5,0.5){$\vdots$};
  \node(10) at (-3.5,0.2){$\bullet$};

\draw[dotted] (-3,0.5) .. controls (-3.2,1) and (-2.8,1) .. (-3,0.5);
\draw[dotted] (-2,0.5) .. controls (-2.2,1) and (-1.8,1) .. (-2,0.5);
\draw[dotted] (-1,0.5) .. controls (-1.2,1) and (-0.8,1) .. (-1,0.5);
\draw[dotted] (0,0.5) .. controls (-0.2,1) and (0.2,1) .. (0,0.5);
\draw[dotted] (1,0.5) .. controls (0.8,1) and (1.2,1) .. (1,0.5);
\draw[dotted] (2,0.5) .. controls (1.8,1) and (2.2,1) .. (2,0.5);

  \node(101) at (-2.8,0.2) {$x\!_{_{\maxNNn}}$};
  \node(102) at (-1.9,0.2) {$x\!_{_{\maxNNn\!+\!1}}$};
  \node(103) at (-0.9,0.2) {$x\!_{_{\maxNNn\!+\!2}}$};
  \node(104) at (0.1,0.2) {$x\!_{_{\maxNNn\!+\!3}}$};
  \node(105) at (1.1,0.2) {$x\!_{_{\maxNNn\!+\!4}}$};
  \node(106) at (2.1,0.2) {$x\!_{_{\maxNNn\!+\!5}}$};

\tikzstyle{every loop}=   [-, shorten >=.5pt]

  \path[-] 
	(-1) edge (1)
	(0)  edge (1)
	(1)  edge (2)	 
	(2)  edge (3)	
	(3)  edge (4)
	(4)  edge (5)	
	(5)  edge (6)
	(6)  edge (7)
	(8)  edge (1)
	(10) edge (1);

\draw[dotted] (-4,1.2) rectangle (-2.5,-0.2);
\node at (-3.25, 1.5) {$G$};

\end{tikzpicture}
\end{center}
\caption{}\label{Figure:TheQuotientGraphOfBoundedType}
\end{figure}

The quotient graph $X$ is linear from the vertex of maximal type $\maxNN+1$.
In the figure, the vertex labeled by $x_k$ is of maximal type $k$.

\end{prop}

In the rest of the section, we provide examples of quasi-Sturmian colorings.
By Theorem~\ref{Thm:AnIffConditionOfTheQuotientGraphOfBoundedType}, the following examples are quasi-Sturmian colorings.

\begin{exam}[quasi-Sturmian coloring which is not a Sturmian coloring] 
Let $c\ge 3$ and $\mathcal{A}=\{a_1, \cdots, a_c\}$. Consider a coloring whose quotient graph is as follows.

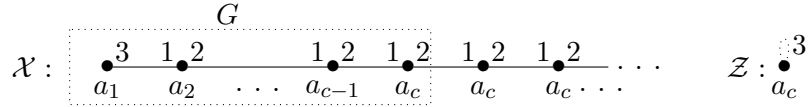
\begin{center}
\begin{tikzpicture}[every loop/.style={}]
  \tikzstyle{every node}=[inner sep=-1pt]

  \node(0) at (-6,1) {$\cX:$};
  \node(1) at (-5,1) {$\bullet$};
  \node(2) at (-4,1) {$\bullet$};
  \node(4) at (-2,1) {$\bullet$};
  \node (5) at (-1,1) {$\bullet$};
  \node (6) at  (0,1) {$\bullet$};
  \node (7) at  (1,1) {$\bullet$};
  \node (8) at (2,1) {$~.~.~.$};

  \node(11) at (-5,0.7) {$a_1$};
  \node(12) at (-4,0.7) {$a_2$};
  \node(13) at (-3,0.7) {$.~.~.$};
  \node(14) at (-2,0.7) {$a_{c-1}$};
  \node (15) at (-1,0.7) {$a_c$};
  \node (16) at  (0,0.7) {$a_c$};
  \node (17) at  (1,0.7) {$a_c$};
  \node (18) at (1.5,0.7) {$~.~.~.$};

  \node(101) at (-4.8,1.2) {$3$};
  \node(102) at (-3.8,1.2) {$2$};
  \node(104) at (-1.8,1.2) {$2$};
  \node (105) at (-0.8,1.2) {$2$};
  \node (106) at  (0.2,1.2) {$2$};
  \node (107) at  (1.2,1.2) {$2$};

  \node(102) at (-4.2,1.2) {$1$};
  \node(104) at (-2.2,1.2) {$1$};
  \node (105) at (-1.2,1.2) {$1$};
  \node (106) at  (-0.2,1.2) {$1$};
  \node (107) at  (0.8,1.2) {$1$};
  
  \node at (3.5,1) {$\cZ:$};
  \node at (4,1) {$\bullet$};
  \node at (4.2,1.3) {$3$};
  \node at (4,0.7) {$a_c$};
	\draw[dotted] (4,1) .. controls (3.8,1.5) and (4.2,1.5) .. (4,1);

\node at (-3.4, 1.7) {$G$};

\tikzstyle{every loop}=   [-, shorten >=.5pt]

  \path[-] 
      (1)  edge  (2)	 
      (2)  edge  (4)	
      (4)  edge  (5)	
      (5)  edge  (6)
	(6)  edge (7)
	(7)  edge (8);

\draw[dotted] (-5.5,1.5) rectangle (-0.7,0.5);
\end{tikzpicture}
\end{center}

It has factor complexity 
\begin{equation*}
b(n)=\left\{\begin{array}{lll}1, &\mathrm{if} \emph{ } n=-1 \\
                                         n+c, & \mathrm{if}\emph{ } n\ge 0.
                     \end{array}\right.
\end{equation*}
$\initN=\maxNN=0$
\end{exam}

The alphabet $\mathcal{A}$ is $\{ \bullet, \circ$,  {\raisebox{0.39ex}{\boldotimes}}$\}$ for the following examples.

\begin{exam}[quasi-Sturmian coloring whose quotient graph is not a geodesic ray]
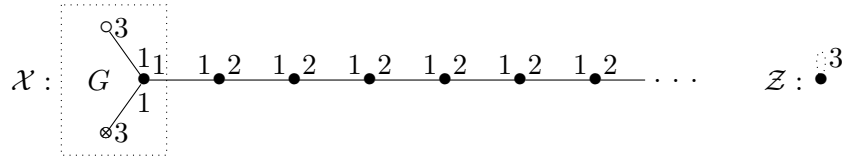
\begin{flushleft}
\end{flushleft}
\begin{center}
\begin{tikzpicture}[every loop/.style={}]
  \tikzstyle{every node}=[inner sep=-1pt]

  \node(-2) at (-6.5,1) {$\cX:$};
  \node(-1) at (-5.5, 1.7){$\circ$};
  \node(0) at (-5.5, 0.3){\boldotimes};
  \node(1) at (-5,1) {$\bullet$};
  \node(2) at (-4,1) {$\bullet$};
  \node(3) at (-3,1) {$\bullet$};
  \node(4) at (-2,1) {$\bullet$};
  \node (5) at (-1,1) {$\bullet$};
  \node (6) at  (0,1) {$\bullet$};
  \node (7) at  (1,1) {$\bullet$};
  \node (8) at (2,1) {$~.~.~.$};

  \node at (-5.3,1.7) {$3$};
  \node at (-5.3,0.3) {$3$};
  \node at (-5,1.3) {$1$};
  \node at (-5,0.7) {$1$};

  \node(101) at (-4.8,1.2) {$1$};
  \node(102) at (-3.8,1.2) {$2$};
  \node(103) at (-2.8,1.2) {$2$};
  \node(104) at (-1.8,1.2) {$2$};
  \node (105) at (-0.8,1.2) {$2$};
  \node (106) at  (0.2,1.2) {$2$};
  \node (107) at  (1.2,1.2) {$2$};

  \node(102) at (-4.2,1.2) {$1$};
  \node(102) at (-3.2,1.2) {$1$};
  \node(104) at (-2.2,1.2) {$1$};
  \node (105) at (-1.2,1.2) {$1$};
  \node (106) at  (-0.2,1.2) {$1$};
  \node (107) at  (0.8,1.2) {$1$};

\node at (-5.6, 1) {$G$};

  \node at (3.5,1) {$\cZ:$};
  \node at (4,1) {$\bullet$};
  \node at (4.2,1.3) {$3$};
	\draw[dotted] (4,1) .. controls (3.8,1.5) and (4.2,1.5) .. (4,1);

\tikzstyle{every loop}=   [-, shorten >=.5pt]

  \path[-] 
     (-1) edge (1)
      (0)  edge (1)
      (1)  edge (2)	 
      (2)  edge (4)	
      (4)  edge (5)	
      (5)  edge (6)
      (6)  edge (7)
      (7)  edge (8);

\draw[dotted] (-6.1,2) rectangle (-4.7,0);
\end{tikzpicture}
\end{center}

The factor complexity is 
\begin{equation*}
b(n)=\left\{\begin{array}{lll}1, &\mathrm{if} \emph{ } n=-1 \\
                                         n+3, & \mathrm{if}\emph{ } n\ge 0
                     \end{array}\right.
\end{equation*}
and $\initN=\maxNN=0$. 
\end{exam}

\begin{exam}[quasi-Sturmian coloring with $\initN\not= 0$]
\begin{flushleft}
\end{flushleft}
\begin{center}
\begin{tikzpicture}[every loop/.style={}]
  \tikzstyle{every node}=[inner sep=-1pt]

  \node(0) at (-5.8,1) {$\cX:$};
  \node(1) at (-5,1) {\boldotimes};
  \node(2) at (-4,1) {$\bullet$};
  \node(3) at (-3,1) {$\circ$};
  \node(4) at (-2,1) {$\bullet$};
  \node (5) at (-1,1) {$\circ$};
  \node (6) at  (0,1) {$\bullet$};
  \node (7) at  (1,1) {$\circ$};
  \node (8) at  (2,1) {$\bullet$};
  \node (9) at  (3,1) {$\circ$};
  \node (10) at (4,1) {$\bullet$};
  \node (11) at (5,1) {$\circ$};
  \node (12) at (6,1) {$\cdots$};

  \node(101) at (-4.8,1.2) {$3$};
  \node(102) at (-3.8,1.2) {$1$};
  \node(103) at (-2.8,1.2) {$1$};
  \node(104) at (-1.8,1.2) {$2$};
  \node (105) at (-0.8,1.2) {$2$};
  \node (106) at  (0.2,1.2) {$1$};
  \node (107) at  (1.2,1.2) {$1$};
  \node (108) at  (2.2,1.2) {$1$};
  \node (109) at  (3.2,1.2) {$2$};
  \node (110) at (4.2,1.2) {$1$};
  \node (111) at (5.2,1.2) {$1$};

  \node(102) at (-4.2,1.2) {$1$};
  \node(103) at (-3.2,1.2) {$2$};
  \node(104) at (-2.2,1.2) {$1$};
  \node (105) at (-1.2,1.2) {$1$};
  \node (106) at  (-0.2,1.2) {$1$};
  \node (107) at  (0.8,1.2) {$2$};
  \node (108) at  (1.8,1.2) {$2$};
  \node (109) at  (2.8,1.2) {$1$};
  \node (110) at (3.8,1.2) {$1$};
  \node (111) at (4.8,1.2) {$2$};

\draw[dotted] (-4,1) .. controls (-4.2,1.5) and (-3.8,1.5) .. (-4,1);
\draw[dotted] (0,1) .. controls (-0.2,1.5) and (0.2,1.5) .. (0,1);
\draw[dotted] (4,1) .. controls (3.8,1.5) and (4.2,1.5) .. (4,1);

  \node at (-4,1.5) {$1$};
  \node at (0,1.5) {$1$};
  \node at (4,1.5) {$1$};
  
  \node at (-3,0) {$\cZ:$};
  \node(1000) at (-2,0) {$\bullet$};
  \node(1001) at (-1,0) {$\circ$};
  \node(1002) at (0,0) {$\bullet$};
  \node(1003) at (1,0) {$\circ$};
  \node(1004) at (2,0) {$\bullet$};

\draw[dotted] (-2,0) .. controls (-2.2,0.5) and (-1.8,0.5) .. (-2,0);
\draw[dotted] (2,0) .. controls (1.8,0.5) and (2.2,0.5) .. (2,0);

  \node at (-1.8,0.2) {$2$};
  \node at (-0.8,0.2) {$1$};
  \node  at  (0.2,0.2) {$2$};
  \node at  (1.2,0.2) {$2$};

  \node at (-1.2,0.2) {$2$};
  \node at  (-0.2,0.2) {$1$};
  \node at  (0.8,0.2) {$1$};
  \node at  (1.8,0.2) {$2$};
  
  \node at (-2,0.5){$1$};
  \node at (2,0.5){$1$};



\tikzstyle{every loop}=   [-, shorten >=.5pt]

  \path[-] 
      (1)  edge  (2)	 
      (2)  edge  (3)
      (3)  edge  (4)	
      (4)  edge  (5)	
      (5)  edge  (6)
	(6)  edge (7)
	(7)  edge (8)
	(8)  edge (9)
	(9)  edge  (10)
	(10) edge (11)
	(11)  edge  (12)
	(1000) edge (1001)
	(1001) edge (1002)
	(1002) edge (1003)
	(1003) edge (1004);


\draw[dotted] (-5.3,1.7) rectangle (-2.6,0.7);
\node at (-4.5, 1.9) {$G$};
\end{tikzpicture}
\end{center}

Its factor complexity is
\begin{equation*}
b(n)=\left\{\begin{array}{lll}1, &\mathrm{if} \emph{ } n=-1 \\
                                         3, &\mathrm{if} \emph{ } n=0 \\
                                         n+4, & \mathrm{if}\emph{ } n\geq 1
                     \end{array}\right.
\end{equation*}
and $\initN=\maxNN=1$.
\end{exam}

\begin{exam}[an example with $\initN\not=0$ and the quotient graph is not a geodesic ray]
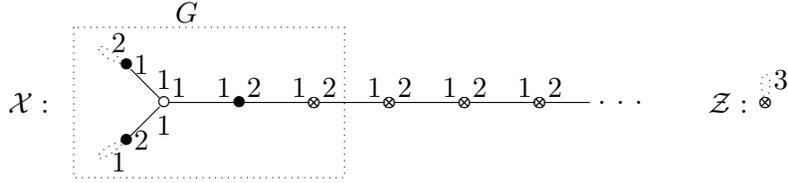
\begin{flushleft}
\end{flushleft}

\begin{center}
\begin{tikzpicture}[every loop/.style={}]
  \tikzstyle{every node}=[inner sep=-1pt]
  \node(-2) at (-4.8,1.5){$\cX :$};
  \node(-1) at (-3.5, 2){$\bullet$};
  \node(0) at (-3.5, 1){$\bullet$};
  \node(1) at (-3,1.5) {$\circ$};
  \node(2) at (-2,1.5) {$\bullet$};
  \node(3) at (-1,1.5) {\boldotimes};
  \node(4) at (0,1.5) {\boldotimes};
  \node (5) at (1,1.5) {\boldotimes};
  \node (6) at  (2,1.5) {\boldotimes};
  \node (7) at  (3,1.5) {$~.~.~.$};

  \node at (-3.3,2) {$1$};  \node at (-3.3,1) {$2$};
  \node at (-3,1.8) {$1$};  \node at (-3,1.2) {$1$};
  \node at (-3.6, 2.3){$2$}; \node at (-3.6, 0.7) {$1$};

\draw[dotted] (-3.5,2) .. controls (-4,2.1) and (-4,2.5) .. (-3.5,2);
\draw[dotted] (-3.5,1) .. controls (-4,0.5) and (-4,0.9) .. (-3.5,1);

  \node(101) at (-2.8,1.7) {$1$};
  \node(102) at (-1.8,1.7) {$2$};
  \node(103) at (-0.8,1.7) {$2$};
  \node(104) at (0.2,1.7) {$2$};
  \node (105) at (1.2,1.7) {$2$};
  \node (106) at  (2.2,1.7) {$2$};

  \node(102) at (-2.2,1.7) {$1$};
  \node(102) at (-1.2,1.7) {$1$};
  \node (107) at  (-0.2,1.7) {$1$};
  \node(104) at (0.8,1.7) {$1$};
  \node (105) at (1.8,1.7) {$1$};
  
  \node at (4.5,1.5) {$\cZ:$};
  \node at (5,1.5) {\boldotimes};
  \node at (5.2,1.8) {$3$};
	\draw[dotted] (5,1.5) .. controls (4.8,2) and (5.2,2) .. (5,1.5);

\tikzstyle{every loop}=   [-, shorten >=.5pt]

  \path[-] 
	(-1) edge (1)
	(0)  edge (1)
	(1)  edge (2)	 
	(2)  edge (3)	
	(3)  edge (4)
	(4)  edge (5)	
	(5)  edge (6)
	(6)  edge (7);

\draw[dotted] (-4.2,2.5) rectangle (-0.6,0.5);
\node at (-2.7, 2.7) {$G$};

\end{tikzpicture}
\end{center}

\begin{center}
\begin{tikzpicture}[every loop/.style={}]
  \tikzstyle{every node}=[inner sep=-1pt]

\end{tikzpicture}
\end{center}

The factor complexity is 
\begin{equation*}
b(n)=\left\{\begin{array}{lll}3, &\mathrm{if} \emph{ } n=0 \\
                                         n+5, & \mathrm{if}\emph{ } n\geq 1
                     \end{array}\right.
\end{equation*}
and $\initN=\maxNN=1$

\end{exam}

\begin{exam}[an example with a cycle in the compact part $G$]
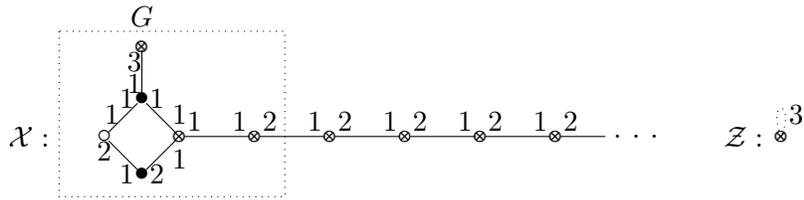
\begin{flushleft}
\end{flushleft}

\begin{center}
\begin{tikzpicture}[every loop/.style={}]
  \tikzstyle{every node}=[inner sep=-1pt]
  \node(-4) at (-5, 1.5){$\cX :$};
  \node(-3) at (-3.5,2.7){\boldotimes};
  \node(-2) at (-4,1.5){$\circ$};
  \node(-1) at (-3.5, 2){$\bullet$};
  \node(0) at (-3.5, 1){$\bullet$};
  \node(1) at (-3,1.5) {\boldotimes};
  \node(2) at (-2,1.5) {\boldotimes};
  \node(3) at (-1,1.5) {\boldotimes};
  \node(4) at (0,1.5) {\boldotimes};
  \node (5) at (1,1.5) {\boldotimes};
  \node (6) at  (2,1.5) {\boldotimes};
  \node (7) at  (3,1.5) {$~.~.~.$};

  \node at (-3.3,2) {$1$};  \node at (-3.3,1) {$2$};
  \node at (-3,1.8) {$1$};  \node at (-3,1.2) {$1$};
  \node at (-3.6, 2.2){$1$}; \node at (-3.7, 1) {$1$};
  \node at (-4,1.3){$2$}; \node at (-3.9,1.8){$1$}; \node at (-3.7, 2){$1$}; \node at (-3.6,2.5){$3$};

  \node(101) at (-2.8,1.7) {$1$};
  \node(102) at (-1.8,1.7) {$2$};
  \node(103) at (-0.8,1.7) {$2$};
  \node(104) at (0.2,1.7) {$2$};
  \node (105) at (1.2,1.7) {$2$};
  \node (106) at  (2.2,1.7) {$2$};

  \node(102) at (-2.2,1.7) {$1$};
  \node(102) at (-1.2,1.7) {$1$};
  \node (107) at  (-0.2,1.7) {$1$};
  \node(104) at (0.8,1.7) {$1$};
  \node (105) at (1.8,1.7) {$1$};
  
  \node at (4.5,1.5) {$\cZ:$};
  \node at (5,1.5) {\boldotimes};
  \node at (5.2,1.8) {$3$};
	\draw[dotted] (5,1.5) .. controls (4.8,2) and (5.2,2) .. (5,1.5);

\tikzstyle{every loop}=   [-, shorten >=.5pt]

  \path[-] 
     (-3) edge (-1)
     (-1) edge (-2)
      (0) edge (-2)
     (-1) edge (1)
      (0)  edge (1)
      (1)  edge (2)	 
      (2)  edge (3)
      (3)  edge (4)	
      (4)  edge (5)	
      (5)  edge (6)
      (6)  edge (7);

\draw[dotted] (-4.6,2.9) rectangle (-1.6,0.7);
\node at (-3.5, 3.1) {$G$};
\end{tikzpicture}
\end{center}

It has the factor complexity
\begin{equation*}
b(n)=\left\{\begin{array}{lll}3, &\mathrm{if} \emph{ } n=0 \\
                                          5, &\mathrm{if} \emph{ } n=1 \\
                                         n+5, & \mathrm{if}\emph{ } n\geq 2
                     \end{array}\right.
\end{equation*}
and $\initN=\maxNN=2$.
\end{exam}


\begin{exam}[an example with $\initN\not=\maxNN$]
\begin{flushleft}
\end{flushleft}

\begin{center}
\begin{tikzpicture}[every loop/.style={}]
  \tikzstyle{every node}=[inner sep=-1pt]

  \node(0) at (-5.8,1) {$\cX:$};
  \node(1) at (-5,1) {\boldotimes};
  \node(2) at (-4,1) {$\bullet$};
  \node(3) at (-3,1) {$\circ$};
  \node(4) at (-2,1) {$\bullet$};
  \node (5) at (-1,1) {\boldotimes};
  \node (6) at  (0,1) {$\bullet$};
  \node (7) at  (1,1) {$\circ$};
  \node (8) at  (2,1) {$\bullet$};
  \node (9) at  (3,1) {\boldotimes};
  \node (10) at (4,1) {$\bullet$};
  \node (11) at (5,1) {$\circ$};
  \node (12) at (6,1) { $\cdots$};

  \node(101) at (-4.8,1.2) {$3$};
  \node(102) at (-3.8,1.2) {$1$};
  \node(103) at (-2.8,1.2) {$1$};
  \node(104) at (-1.8,1.2) {$1$};
  \node (105) at (-0.8,1.2) {$1$};
  \node (106) at  (0.2,1.2) {$1$};
  \node (107) at  (1.2,1.2) {$1$};
  \node (108) at  (2.2,1.2) {$1$};
  \node (109) at  (3.2,1.2) {$1$};
  \node (110) at (4.2,1.2) {$1$};
  \node (111) at (5.2,1.2) {$1$};

  \node(102) at (-4.2,1.2) {$2$};
  \node(103) at (-3.2,1.2) {$2$};
  \node(104) at (-2.2,1.2) {$2$};
  \node (105) at (-1.2,1.2) {$2$};
  \node (106) at  (-0.2,1.2) {$2$};
  \node (107) at  (0.8,1.2) {$2$};
  \node (108) at  (1.8,1.2) {$2$};
  \node (109) at  (2.8,1.2) {$2$};
  \node (110) at (3.8,1.2) {$2$};
  \node (111) at (4.8,1.2) {$2$};
  
  \node at (-1,-0.5){$\cZ:$};
  \node(1000) at (0,0){$\bullet$};
  \node(1001) at (1,0){$\circ$};
  \node(1002) at (1,-1){$\bullet$};
  \node(1003) at (0,-1){\boldotimes};
  
  \node at (-0.15,-0.2){$2$};
  \node at (0.15,0.2){$1$};
  \node at (0.85,0.2){$2$};
  \node at (1.15,-0.2){$1$};
  \node at (1.15,-0.8){$2$};
  \node at (0.85,-1.2){$1$};
  \node at (0.15, -1.2){$2$};
  \node at (-0.15,-0.8){$1$};






\tikzstyle{every loop}=   [-, shorten >=.5pt]

  \path[-] 
      (1)  edge  (2)	 
      (2)  edge  (3)
      (3)  edge  (4)	
      (4)  edge  (5)	
      (5)  edge  (6)
	(6)  edge (7)
	(7)  edge (8)
	(8)  edge (9)
	(9)  edge  (10)
	(10) edge (11)
	(11)  edge  (12)
	(1000) edge (1001)
	(1001) edge (1002)
	(1002) edge (1003)
	(1003) edge (1000)
;


\draw[dotted] (-5.3,1.5) rectangle (-4.6,0.5);
\node at (-4.9, 1.7) {$G$};
\end{tikzpicture}
\end{center}

The factor complexity is 
\begin{equation*}
b(n)=\left\{\begin{array}{lll}1, &\mathrm{if} \emph{ } n=-1 \\
                                          n+3, &\mathrm{if} \emph{ } n\geq 0
                     \end{array}\right.
\end{equation*}
and $\initN=0$, $\maxNN=1$.
\end{exam}

\subsection{Quasi-Sturmian colorings of unbounded type} The quotient graph of a Sturmian coloring of unbounded type is a geodesic ray or an infinite geodesic (see Theorem 3.8 in \cite{KL}). 
In this section, we show that the same property holds for quasi-Sturmian colorings of unbounded type.

\begin{prop} 
\label{Prop:ForUnboundedType,TheNeighboringVerticesHaveAtMost3TypeSets}
For a quasi-Sturmian coloring of unbounded type, the vertices of a $1$-ball have at most three distinct type sets.
\end{prop}

\begin{proof}
Suppose that there are three vertices $u_1$, $u_2$, $u_3$ neighboring $u$ such that $u$, $u_1$, $u_2$, $u_3$ have mutually distinct type sets.
If $n\in\Lambda_u\cap\Lambda_v$ and $n\geq \initN$, then $[\B_n(u)]=[\B_n(v)]$ by the uniqueness of the special $n$-ball.
Thus, for $l\leq n$, $l\in\Lambda_u$ if and only if $l\in\Lambda_v$.

If $\Lambda_u\not=\Lambda_v$ and $\Lambda_u\cap\Lambda_v\not=\phi$, then let $N$ be the maximal element of $\Lambda_u\cap\Lambda_v$.
If $\Lambda_u\cap\Lambda_v=\phi$, let $N=-1$.
Choose such $N$ for each pair of vertices from different classes in $\B_2(u)$ and let $M$ be the maximum of such $N$'s. 
Then, the type sets of two non-equivalent vertices in $\B_2(u)$ intersected with $\{M+1, M+2, \cdots\}$ are all mutually disjoint.

Now let $l> M+1$ be in the type set $\Lambda_u$. Such $l$ exists since the coloring is of unbounded type.
At least one of $u_1,u_2,u_3$ has a type set disjoint from $\{l-1,l,l+1\}$, say $u_i$.
Since $l\in\Lambda_u$, there is $v$ such that  $[\B_l(u)]=[\B_l(v)]$ but $[\B_{l+1}(u)]\not=[\B_{l+1}(v)]$.
Let $f:\B_{l}(u)\rightarrow\B_{l}(v)$ be a color-preserving isometry.
Then $[\B_{l-1}(u_i)]=[\B_{l-1}(f(u_i))]$.

Let $p=\min\{k\geq l-1 : k\in\Lambda_{u_i}\}$. 
Since $p>l+1$, $[\B_{l-1}(u_i)]$ has a unique extension to $[\B_{p}(u_i)]$.
Thus, $[\B_{p}(u_i)]$ and $[\B_{p}(f(u_i))]$ are equivalent by a color-preserving isometry $g$.
Since $[\B_{p-1}(g^{-1}(v))]=[\B_{p-1}(v)]$ and $p-1>l$,
$[\B_{l}(g^{-1}(v))]=[\B_{l}(v)]=[\B_{l}(u)]$
and 
$[\B_{l+1}(g^{-1}(v))]=[\B_{l+1}(v)]\not=[\B_{l+1}(u)]$.
Thus, $g^{-1}(v)\not=u$ and $\Lambda_{g^{-1}(v)}\cap\Lambda_{u}$ contains $l>M+1$.
However, since $\metricd(g^{-1}(v),u)\le 2$, it contradicts that $\Lambda_{g^{-1}(v)}\cap\Lambda_{u}\cap \{M+1,M+2,\cdots\}$ is empty.
\end{proof}

Let $(\cT,\phi)$ be a quasi-Sturmian coloring of a tree and $\cX=(X,i)$ be its quotient graph.
If two vertices $u$, $v$ have the same type set, they have the same colored $n$-balls for every $n$, i.e. $u,v$ are equivalent (see Lemma 2.4 in \cite{KL}).
By Proposition~\ref{Prop:ForUnboundedType,TheNeighboringVerticesHaveAtMost3TypeSets}, there are at most $2$ adjacent vertices of each vertex $x\in VX$.

For a quasi-Sturmian coloring of unbounded type, we define $G$ as the set of vertices which has only one adjacent vertex in $X$.
Since factor complexity of $\phi$ is unbounded, $X$ is an infinite graph.
Since $X$ is connected, $G$ is empty or $G$ has a single element.
Thus, we obtain the following characterization of the quotient graphs of quasi-Sturmian colorings of trees.

\begin{thm}\label{Thm:TheQuotientGraphOfQS}
If $\phi$ is a quasi-Sturmian coloring, then the quotient graph is one of the following graphs.

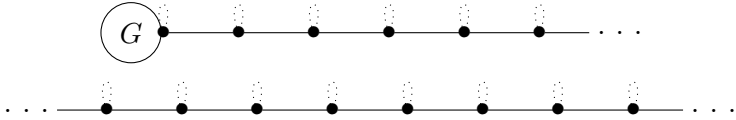
\begin{figure}[h]
\begin{center}
\begin{tikzpicture}[every loop/.style={}]
  \tikzstyle{every node}=[inner sep=-1pt]
  \node(1)  at (-3   ,0.5) {$\bullet$};
  \node(2)  at (-2   ,0.5) {$\bullet$};
  \node(3)  at (-1   ,0.5) {$\bullet$};
  \node(4)  at ( 0   ,0.5) {$\bullet$};
  \node(5)  at ( 1   ,0.5) {$\bullet$};
  \node(6)  at ( 2   ,0.5) {$\bullet$};
  \node(7)  at ( 3   ,0.5) {$~.~.~.$};

	\draw[dotted] (-3,0.5) .. controls (-3.2,1) and (-2.8,1) .. (-3,0.5);
	\draw[dotted] (-2,0.5) .. controls (-2.2,1) and (-1.8,1) .. (-2,0.5);
	\draw[dotted] (-1,0.5) .. controls (-1.2,1) and (-0.8,1) .. (-1,0.5);
	\draw[dotted] (0,0.5) .. controls (-0.2,1) and (0.2,1) .. (0,0.5);
	\draw[dotted] (1,0.5) .. controls (0.8,1) and (1.2,1) .. (1,0.5);
	\draw[dotted] (2,0.5) .. controls (1.8,1) and (2.2,1) .. (2,0.5);

\tikzstyle{every loop}=   [-, shorten >=.5pt]

  \path[-] 
	(1)  edge (2)	 
	(2)  edge (3)	
	(3)  edge (4)
	(4)  edge (5)	
	(5)  edge (6)
	(6)  edge (7)
	;
	
\draw (-3.43,0.5) circle (.4cm);
\node at (-3.43, .5) {$G$};

\end{tikzpicture}
\end{center}

\begin{tikzpicture}[every loop/.style={}]
  \tikzstyle{every node}=[inner sep=-1pt]
  \node(-2) at (-6, 0.5) {$.~.~.~$};
  \node(-1) at (-5, 0.5) {$\bullet$};
  \node(0)  at (-4, 0.5) {$\bullet$};
  \node(1)  at (-3, 0.5) {$\bullet$};
  \node(2)  at (-2, 0.5) {$\bullet$};
  \node(3)  at (-1, 0.5) {$\bullet$};
  \node(4)  at ( 0, 0.5) {$\bullet$};
  \node(5)  at ( 1, 0.5) {$\bullet$};
  \node(6)  at ( 2, 0.5) {$\bullet$};
  \node(7)  at ( 3, 0.5) {$~.~.~.$};

	\draw[dotted] (-5,0.5) .. controls (-5.2,1) and (-4.8,1) .. (-5,0.5);
	\draw[dotted] (-4,0.5) .. controls (-4.2,1) and (-3.8,1) .. (-4,0.5);
	\draw[dotted] (-3,0.5) .. controls (-3.2,1) and (-2.8,1) .. (-3,0.5);
	\draw[dotted] (-2,0.5) .. controls (-2.2,1) and (-1.8,1) .. (-2,0.5);
	\draw[dotted] (-1,0.5) .. controls (-1.2,1) and (-0.8,1) .. (-1,0.5);
	\draw[dotted] (0,0.5) .. controls (-0.2,1) and (0.2,1) .. (0,0.5);
	\draw[dotted] (1,0.5) .. controls (0.8,1) and (1.2,1) .. (1,0.5);
	\draw[dotted] (2,0.5) .. controls (1.8,1) and (2.2,1) .. (2,0.5);

\tikzstyle{every loop}=   [-, shorten >=.5pt]
  \path[-]
   (-2) edge (-1)
   (-1) edge (0)
	(0)  edge (1)
	(1)  edge (2)	 
	(2)  edge (3)	
	(3)  edge (4)
	(4)  edge (5)	
	(5)  edge (6)
	(6)  edge (7)
	;
\end{tikzpicture}
\end{figure}
\end{thm}

\section{Evolution of factor graphs}\label{Sec_CF}

In this section, we look into quasi-Sturmian colorings of unbounded type in details.
Let us begin by explaining an induction algorithm for quasi-Sturmian colorings of bounded type. 
As in \cite{KL2}, for $n \geq \initN$, $S_n$ denotes a unique special $n$-ball, $C_n$ denotes a centered $n$-ball of $S_{n+1}$, and $A_{n+1}$, $B_{n+1}$ denote two types of extensions of $S_n$.
For a class of $n$-balls $B=[\B_n(x)]$, denote the class of $[\B_{n+1}(x)]$ by $\overline{B}$ and the class of $[\B_{n-1}(x)]$ by $\underline{B}$.
Note that if $B$ is not special, then $\overline{B}$ is well-defined.

Recall from the introduction that for a given quasi-Sturmian coloring $\phi$, for $n\geq {\initN+1}$, the \emph{factor graph} $\mathcal{G}_n$ has $\mathbb{B}_\phi(n)$ as its vertex set. There is an edge between two colored $n$-balls $D$, $E$ if there exist $n$-balls centered at $x$, $y$ in the classes $D$, $E$, respectively, such that $\mathbf{d}(x,y)$=1. 

\subsection{Prelimiary}
Now, we gether preliminaries of cyclic quasi-Sturmian colorings.

\begin{defn}
We say that $D$ is \emph{weakly adjacent} to $E$ if there exist $v,w\in V\cT$ such that $\metricd(v,w)=1$ and $[\B_n(v)]=D$ and $[\B_m(w)]=E$ for some $n,m$.

We also say that $D$ is \emph{strongly adjacent} to $E$ if for any $\B_n(x)$ in the class $D$, there exists a vertex $y$ such that $\B_m(y) \in E$ and $\metricd(x,y)=1$. 
If $D$ is strongly adjacent to $E$ and vice versa, then we say that $D$ and $E$ are strongly adjacent.
\end{defn}

We remark the following fact.
If $[\B_{n+1}(u)]=[\B_{n+1}(v)]$ and $[\B_{n+2}(u)]\not=[\B_{n+2}(v)]$, then there exist neighboring vertices $u'$ and $v'$ of $u$ and $v$, respectively, such that $[\B_{n}(u')]=[\B_{n}(v')]$ and $[\B_{n+1}(u')]\not=[\B_{n+1}(v')]$ (see Lemma 2.11 in \cite{KL} for details).
Thus, $S_{n+1}$ is strongly adjacent to $S_n$ for $n \ge \initN$.

\begin{lem}
\label{lem2.3_qs}
Let $(\cT,\phi)$ be a quasi-Sturmian coloring and $n\geq \initN$. 

\begin{enumerate}[(1)]

\item\label{lem2.3_qs(1)} We can choose ${\{A_n\}_{n\ge {\initN+1}}}$, ${\{B_n\}_{n\ge {\initN+1}}}$ so that $A_{n+1}$, $B_{n+1}$ are strongly adjacent to ${A_n}$, ${B_n}$, respectively. Moreover, $A_{n+1} $, $B_{n+1}$ are uniquely determined if we give the condition that $A_{n+1}$ contains more balls of the class $A_{n}$ than $B_{n+1}$ does.

\item\label{lem2.3_qs(2)} For each vertex $x$ in $\cT-\widetilde{G}$ and $n \ge \initN+1$, 
the $n$-balls with centers adjacent to $x$ belong to at most two classes of $n$-balls apart from $[\mathcal{B}_n(x)]$. 
Thus, for any class $D\not= S_n$ of $n$-balls with centers in $\cT-\widetilde{G}$, 
each vertex of $\mathcal{G}_n$ has degree at most $2$.

\item\label{lem2.3_qs(3)} If $A_n \ne S_n$ (respectively $B_n \ne S_n$), then $A_n$ (respectively $B_n$) is strongly adjacent to $S_n$.

\item\label{lem2.3_qs(4)} The two classes $S_n, C_n$ are strongly adjacent.

\end{enumerate}
\end{lem}

\begin{proof}
(1) Since $S_{n+1}$ is strongly adjacent to $S_n$ for $n \ge \initN$, we can choose sequences $\{A_n\}_{n \ge \initN+1}$, $\{B_n\}_{n \ge \initN+1}$ such that  $A_{n+1}$, $B_{n+1}$ are strongly adjacent to ${A_n}$, ${B_n}$, respectively, and $A_{n+1}$ contains more balls of class $A_n$ than $B_{n+1}$ does. 
Then, such an inductive choice is unique. 

It is sufficient to show that $A_{n+1}$ , $B_{n+1}$ cannot contain the same number of balls of class $A_n$. 
Denote by $i(D,E)$ the number of $n$-balls of class $E$ which are contained in the colored $(n+1)$-ball $D$.
Then, we have
$$
i(S_n,S_{n-1})= i(A_{n+1}, A_n) +i(A_{n+1}, B_n) = i(B_{n+1},A_n)+i(B_{n+1}, B_n).
$$
If  $i(A_{n+1},A_n)=i(B_{n+1},A_n)$, then $i(A_{n+1},B_n)=i(B_{n+1},B_n)$. Since each $(n-1)$-ball not $S_{n-1}$ has a unique extension to  {an} $n$-ball, we have $A_{n+1}=B_{n+1}$.

(2) By Theorem ~\ref{Thm:TheQuotientGraphOfQS}, there are at most two congruent classes of vertices adjacent to any given vertex $x$ in $\cT-\widetilde{G}$, apart from itself. 
In other words, the number of classes of $n$-balls with  {centers} adjacent to $x$ in $\cT-\widetilde{G}$ is also at most two, apart from $[\B_n(x)]$. 
For  {any class $D\not=S_n$ of $n$-balls with centers} in $\cT-\widetilde{G}$, $D$ has a unique extension to  {$(n+1)$-balls.
Thus, there are at most two classes of $n$-balls which are weakly adjacent to $D$.}

(3) By (\ref{lem2.3_qs(1)}), $A_n$ is weakly adjacent to $\underline{A_{n+1}}=S_n$. If $A_n \ne S_n$, then $A_n$ has a unique extension to $\overline{A_n}$. Thus, $A_n$ is strongly adjacent to $S_n$.

(4)  {
Since $S_{n+1}$ is strongly adjacent to $S_n$ for $n \ge \initN$,
$C_n$ is strongly adjacent to $S_n$.
If $S_n=C_n$, we are done.
} 
 {Suppose that $S_n \ne C_n$.} 
Note that $A_{n+1} \ne S_{n+1}$ and $B_{n+1} \ne S_{n+1}$. 
By (\ref{lem2.3_qs(3)}), both $A_{n+1}, B_{n+1}$ are strongly adjacent to $S_{n+1}$. Hence, $S_n$ is strongly adjacent to $C_n$.
\end{proof}

 We will specify the choice of $A_{\initN+1}$ from the two extension of $S_{\initN}$ for acyclic quasi-Sturmian colorings later.

\begin{lem}
\label{lem2.4_qs}
Let $\phi$ be a quasi-Sturmian coloring and $n$ be greater than $\initN$.
 {Let $D$ be a colored $n$-ball other than $A_n$, $B_n$ and $S_n$.
Assume that $S_n$ and $D$ are weakly adjacent.
Then, we have that}

\begin{enumerate}[(1)]

\item\label{lem2.4_qs(1)} 
 {the special ball $S_n$ and $D$ are strongly adjacent, and}
\item\label{lem2.4_qs(2)} if  {
$D\ne C_n$}, then $S_n \ne C_n$.

\end{enumerate}

\end{lem}

\begin{proof}
(1) Since $S_n$ is weakly adjacent to $D$,  {$S_n$ is strongly adjacent to $\underline{D}$. 
By the assumption, $\underline{D}\ne S_{n-1}$, i.e. $\underline{D}$ is uniquely extended to $D$.
Thus, $S_n$ and $D$ are strongly adjacent.}

(2) Assume that $D\ne C_n$. If $S_n=C_n$, then  {either $S_{n+1}=A_{n+1}$ or $S_{n+1}=B_{n+1}$. 
By Lemma~\ref{lem2.3_qs} (\ref{lem2.3_qs(3)}), 
$A_{n+1}$ and $B_{n+1}$ are weakly adjacent.
Since $A_n\ne S_n$ or $B_n\ne S_n$, we may assume that $A_n\ne S_n$.
We obtain that $A_{n+1}$ is strongly adjacent to $D$ by (\ref{lem2.4_qs(1)}) and to $A_n$ by Lemma ~\ref{lem2.3_qs} (\ref{lem2.3_qs(1)}).
Then, $A_{n+1}$ is weakly adjacent to $B_{n+1}$, $\overline{D}$ and $\overline{A_n}$.
Since $A_{n+1}$, $B_{n+1}$, $\overline{D}$ and $\overline{A_n}$ are mutually distinct,
it contradicts Lemma~\ref{lem2.3_qs} (\ref{lem2.3_qs(2)}).}
\end{proof}


\begin{prop}\label{Prop:EvolutionOfAcyclicQS}
If there are two vertices of degree at least three in $\G_n$ for some $n > \initN$, then the quasi-Sturmian coloring $(\cT,\phi)$ is of bounded type.
\end{prop}

\begin{proof}
If $\phi$ is of unbounded type, $S_n$ is the unique vertex adjacent to distinct three classes of $n$-balls in $\G_n$ by Lemma ~\ref{lem2.3_qs} (\ref{lem2.3_qs(2)}). 
Thus, there is at most one vertex of degree at least three in $\G_n$.
\end{proof}


\begin{defn}
A quasi-Sturmian coloring is $\emph{cyclic}$ if there is a cycle containing $S_n$ in $\mathcal{G}_n$ for some $n>\initN$. If not, we say that a quasi-Sturmian coloring is $\emph{acyclic}$.
\end{defn}
\begin{lem}\label{lem:bdd_1}
Suppose that $\G_n$ has a cycle whose lift in $X$ is not contained in $G$ for some $n \ge \initN+1$. The following statements hold.
\begin{enumerate}

\item\label{lem:bdd_1(1)} The special ball $S_n$ is in the cycle.
\item\label{lem:bdd_1(2)}  {If $D\ne A_n,B_n,C_n,S_n$, then $D$ is not weakly adjacent to $S_n$.}
\end{enumerate}
\end{lem}

\begin{proof}
(1)  {
If $S_n$ is not in the cycle, then it connected to a vertex $D$ of $\G_n$ which is in the cycle.
Thus, the degree of $D$ is greater than $2$.
It contradicts Lemma~\ref{lem2.3_qs} (\ref{lem2.3_qs(2)}).
}

(2) Assume that $S_n$ is weakly adjacent to $D \ne A_n, B_n, C_n,S_n$. 
By Lemma~\ref{lem2.4_qs} (\ref{lem2.4_qs(1)}), $S_n$ is strongly adjacent to $D$. 
By Lemma~\ref{lem2.4_qs} (\ref{lem2.4_qs(2)}) and Lemma~\ref{lem2.3_qs} (\ref{lem2.3_qs(4)}), $S_n \ne C_n$ and $S_n$ is strongly adjacent to $C_n$. 
Hence, $S_{n+1} \ne A_{n+1},B_{n+1} $, and  {the colored $n$-balls appearing in} $A_{n+1},B_{n+1}$ are in $\{C_n, D, S_n\}$ by Lemma~\ref{lem2.3_qs} (\ref{lem2.3_qs(2)}). 
By (\ref{lem:bdd_1(1)}), $C_n\text{ and }D$ are in the cycle. Since $S_n$ is extended to $A_{n+1}\text{ and } B_{n+1}$,  {the degree of $\overline{D}$ is $3$} in $\G_{n+1}$. It is a contradiction.
\end{proof}


\begin{lem}\label{lem:bdd_2}
 {
For $n>\initN$, suppose that $\G_n$ has a cycle whose lift in $X$ is not contained in $G$.
\begin{enumerate}
\item If $C_n$ is not contained in the cycle, then $\G_{n+l}$ has a cycle containing $C_{n+l}$ for some $l \ge 1$.
\item If $C_n=S_n$, then $\G_{n+1}$ has a cycle containing $C_{n+1}$ and $C_{n+1} \ne S_{n+1}$.
\end{enumerate}
}
\end{lem}

\begin{proof}
 {
(1) By Lemma~\ref{lem:bdd_1}, $S_n,A_n,B_n,C_n$ are all distinct.
By Lemma~\ref{lem2.3_qs} (\ref{lem2.3_qs(3)}) and (\ref{lem2.3_qs(4)}), $A_n,B_n,C_n$ are strongly adjacent to $S_n$.
Thus, $A_n$, $B_n$ are in the cycle.
By Lemma~\ref{lem2.3_qs} (\ref{lem2.3_qs(1)}), $A_{n+1}$ and $\overline{A_n}$ are weakly adjacent and $B_{n+1}$ and $\overline{B_n}$ are weakly adjacent.
Since $C_n\ne S_n$, we have $S_{n+1}\ne A_{n+1},B_{n+1}$.
Thus, $A_{n+1}, B_{n+1}$ are strongly adjacent to $S_{n+1}$.
In $\G_{n+1}$, there is a cycle 
$$
[S_{n+1}A_{n+1}\overline{A_n}\cdots\overline{B_n}B_{n+1}S_{n+1}].
$$
If $C_{n+1}$ is equal to one of $A_{n+1}$, $B_{n+1}$, $S_{n+1}$, then we are done.
If $C_{n+1}$ is not equal to $A_{n+1}$, $B_{n+1}$, $S_{n+1}$, then we can apply the above argument again.} 
 {For each $l \ge n$, the number of the vertices of the subgraph of $\G_l$ starting from $S_l$ toward $C_l$ is a decreasing function on $l$. Hence, the above process stops.}

 {
(2) Suppose that $C_n =S_n$.  
By Lemma~\ref{lem:bdd_1}, the cycle is represented by
$[S_nA_n\cdots B_nS_n],$
and $S_n$ is not equal to $A_n$ and $B_n$.
Thus, $\overline{A_n}$ is not equal to $A_{n+1},B_{n+1},S_{n+1}$.
We have either $S_{n+1}=A_{n+1}$ or $S_{n+1}=B_{n+1}$, say $S_{n+1}=A_{n+1}$.
Thus, in $\G_{n+1}$, there is a cycle
$$[S_{n+1}\overline{A_n}\cdots\overline{B_n}B_{n+1}S_{n+1}].$$
By Lemma~\ref{lem:bdd_1} (\ref{lem:bdd_1(2)}), $\overline{A_n}$ is equal to $C_{n+1}$.
Thus, $C_{n+1}\ne S_{n+1}$.
}
\end{proof}



\begin{prop}\label{prop:adjacent}
 { (1)} Let $n \ge \initN+1$.
If there is a ball $D$ which is weakly adjacent to $S_n$ and different from $A_n, B_n, C_n, \text{and } S_n$, then
$\G_{n+1}$ has a cycle containing $\overline D$.    

 {(2)} Any cyclic quasi-Sturmian coloring is of bounded type.
\end{prop}

\begin{proof}
(1) By Lemma ~\ref{lem2.4_qs} (2), $S_n \ne C_n$. Thus, $S_{n+1} \ne A_{n+1}, B_{n+1}$. 
By Lemma ~\ref{lem2.3_qs} (3), both $A_{n+1}, B_{n+1}$ are strongly adjacent to $S_{n+1}$. 
On the other hand, by Lemma ~\ref{lem2.4_qs} (1), $D$ is strongly adjacent to $S_n$.
Thus, $\overline{D}$ is strongly adjacent to $A_{n+1}, B_{n+1}$. 
Hence, $[S_{n+1}A_{n+1}\overline{D}B_{n+1}S_{n+1}]$ is a cycle in $\G_{n+1}$.

(2) Let us prove that a cyclic quasi-Sturmian coloring is of bounded type.
 {
By Lemma~\ref{lem:bdd_2}, there is $n$ such that $S_n\ne C_n$ and $C_n$ is in a cycle
$$
[S_nC_nD\cdots E S_n]
$$ 
of $\G_n$ for some colored $n$-balls $D,E$.
Note that $D$ and $C_n$ are strongly adjacent to each other.
By Lemma ~\ref{lem:bdd_1} (\ref{lem:bdd_1(2)}), $E$ can be $A_n$ or $B_n$, say $A_n$.
Since $S_n\ne C_n$, $A_{n+1}$ and $B_{n+1}$ are stongly adjacent to $S_{n+1}$ respectively.
By Lemma~\ref{lem2.3_qs} (\ref{lem2.3_qs(1)}), $A_{n+1}$ is strongly adjacent to $\overline{A_n}$.
Thus, a cycle of $\G_{n+1}$ is represented by 
$$[A_{n+1}S_{n+1}\overline{D}\cdots\overline{A_n}A_{n+1}].$$
Since $\overline{D}$ cannot be $B_{n+1}$, $\overline{D}=C_{n+1}$.
Note that if a colored $(n+1)$-ball $F$ is contained in the cycle of $\G_{n+1}$, then $\underline{F}$ is in the cycle of $\G_n$. 
Now, we have $S_{n+1}\ne C_{n+1}$ and a cycle containing $C_{n+1}$.
Thus, we can apply above argument for all $m>n$.
}

 {
Since $B_{n+1}$ is not equal to $A_{n+1},S_{n+1},C_{n+1}$, we conclude that $B_{n+1}$ is outside of the cycle of $\G_{n+1}$.
If an extension of ${B_{n+1}}$ to the colored $m$-ball is in the cycle of $\G_m$, then 
$B_{n+1}$ is in the cycle of $\G_{n+1}$.
It is a contradiction.
Thus, any extension of $B_{n+1}$ is not special.
Therefore, for a vertex $u$ such that $[\B_{n+1}(u)]=B_{n+1}$, the vertex $u$ is of bounded type.
} 
\end{proof}

\subsection{Acyclic quasi-Sturmian colorings}

\begin{lem}
\label{A=S=C_less_than_N}
Let $\phi$ be an acyclic quasi-Sturmian coloring.
If $A_N = S_N =C_N$ for some $N > \initN+1$, then $A_n =S_n=C_n$ for all $\initN +1 \le n <N$.
\end{lem}

\begin{proof}
Suppose $A_N=S_N=C_N$ for some $N >\initN+1$. 
 {A colored $n$-ball which is weakly adjacent to $S_n$ is one of $A_n$, $C_n$, $B_n$, $S_n$ for $n\ge \initN+1$ by Proposition ~\ref{prop:adjacent}.}
Thus, $S_N$ is weakly adjacent to only $B_N$ and itself.
Since $|V\G_n| \ge 3$ for $n>1$,
there exists a colored $N$-ball D weakly adjacent to $B_N$. 
 {
The special $(N-1)$-ball $S_{N-1}$ is weakly adjacent to $\underline{D}$.
By $C_{N-1}=S_{N-1}$,we have that $\underline{D}=A_{N-1}$ or $\underline{D}=B_{N-1}$.
If $\underline{D}=A_{N-1}$, then $A_{N}=S_N$ is not weakly adjacent to $A_{N-1}$.}
Hence, $\underline{D}=B_{N-1}$ and $A_{N-1}=C_{N-1}=S_{N-1}$. 
By the same argument, $A_n=S_n=C_n$ for all $\initN+1 \le n < N$. 
\end{proof}

We choose $A_n$ as 
$S_n = C_n = A_n$ if there exists $n >  N_0$ such that $S_n = C_n$ is identical to $A_n$ or $B_n$.
Define
$$
\capitalK = \min\{ n > \initN :  A_n , S_n , C_n  \text{ are not all identical} \}
$$
as in \cite{KL2}.
Note that $K$ may be infinity. 

For an acyclic  quasi-Sturmian coloring, for each $n \ge \capitalK$,
neither $A_{n}, S_{n}, C_{n}$ nor $B_{n}, S_{n}, C_{n}$ are identical.
Therefore, 
the colored $n$-balls $S_n$, $A_n$, $B_n$, $C_n$ satisfy one of the following conditions. 
\begin{itemize}
\item[(I)] $S_n, C_n$ are distinct, but  one of $S_n$, $C_n$ is identical to $A_n$ or $B_n$.
%
%

\item[(II)] $S_n, A_n, B_n, C_n$ are all distinct. 

\item[(III)] $S_n, A_n, B_n$ are distinct, but $S_n = C_n$. 
\end{itemize}

Case (I) is divided into three subcases:
\begin{itemize}
\item[(I-a)] $A_n, B_n, S_n$ are distinct and $C_n = A_n$ or $B_n$,

\item[(I-b)] $A_n, B_n, C_n$ are distinct and $S_n = A_n$ or $B_n$,

\item[(I-c)] $A_n= S_n$, $B_n =C_n$ are distinct,
\end{itemize}

By Lemma~\ref{lem:bdd_1} we deduce that  $S_n$ is a vertex of degree 3 in $\G_n$ for Case (II),
But for Case (I) and (III), $\G_n$ is a linear graph and $S_n$ is of degree 1 or 2.

\begin{prop}\label{graph_develop}
Suppose that $\G_n$ corresponds to Case (I).
Then $S_n$ is a vertex of degree 2 or 1 in $\G$. Thus $\G_n$ is a linear graph.
Let $m$ be the number of vertices connected to $S_n$ through $C_n$. 
Note that $m \ge 1$ since $C_n$ is not identical to $S_n$.
Then we have
$\G_{n+k}$ belongs to Case (II) for all $0 < k < m$ and either 
$\G_{n+m}$ belongs to Case (I) or 
$\G_{n+m}$ belongs to Case (III) and $\G_{n+m+1}$ belongs to Case (I).
\end{prop}


\begin{proof}
If $S_n$ and $C_n$ are distinct, then $\G_n$ belongs to Case (I)  or (II).
We deduce that $S_{n+1}$, $A_{n+1}$, $B_{n+1}$ are distinct.
If $C_n$ is of degree 2, then there exists $D$ neighboring $C_n$ which is not $S_n$.
Thus $\overline D$ is weakly adjacent to $S_{n+1}$ but different from $S_{n+1}, A_{n+1}, B_{n+1}$, 
which is implied that $\overline D = C_{n+1}$, which corresponds Case (II).
In this case the number of vertices connected to $S_{n+1}$ through $C_{n+1}$ decrease by 1.

If $C_n$ is of degree 1, then $m =1$. In this case, $S_{n+1}$ is connected to only $A_{n+1}, B_{n+1}$ two extensions of $S_n$ in $\G_{n+1}$, which implies that 
$C_{n+1} = S_{n+1}$, i.e., Case (III) or $C_{n+1} = A_{n+1}$ or $B_{n+1}$, i.e., Case (I-a).

If $\G_n$ belongs to Case (III), then $S_n = C_n$, thus we have either $S_{n+1} = A_{n+1}$ or $S_{n+1} = B_{n+1}$, say $S_{n+1} = A_{n+1}$.
Since $\overline{A_n}$ is weakly adjacent to $A_{n+1} = S_{n+1}$ and $\overline{A_n}$ cannot be $A_{n+1}$ nor $B_{n+1}$, we deduce that $C_{n+1} = \overline{A_n}$.
Therefore, $\G_{n+1}$ belongs to the Case (I-b).
 
We remark that Case (I-c) can happen only for $n = \capitalK$.
\end{proof}

We denote by $(n_k)$ the subsequence for which $\G_{n_k}$ is of Case (I).
The evolution of $\G_n$ from $n = n_k$ to $n = n_{k+1}$
 is shown in Figure~\ref{cycle1}.

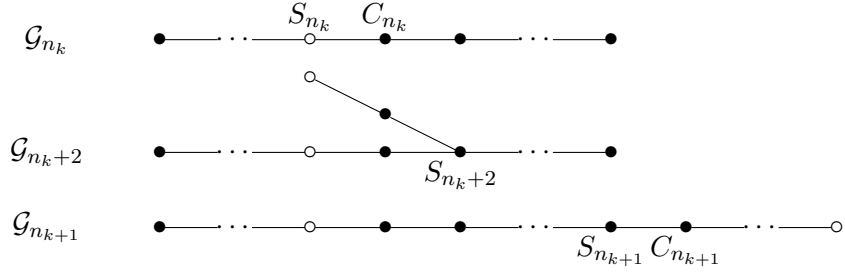
\begin{figure}
\begin{tikzpicture}[every loop/.style={}]
  \tikzstyle{every node}=[inner sep=-1pt]
  \node (4) at (-3.5,1.5) {$\mathcal{G}_{n_k}$};
  \node (0) at (-2,1.5) {$\bullet$};
  \node (1) at (-1,1.5) {$\cdots$};
  \node (3) at  (0,1.5) {$\circ$} node [above=4pt] at (0,1.5) {$S_{n_k}$};
  \node (6) at  (1,1.5) {$\bullet$} node [above=4pt] at (1,1.5) {$C_{n_k}$};
  \node (7) at  (2,1.5) {$\bullet$};
  \node (8) at  (3,1.5) {$\cdots$};
  \node (9) at  (4,1.5) {$\bullet$};

\tikzstyle{every loop}=   [-, shorten >=.5pt]

  \path[-] 
	(0)  edge (1)
	(1)  edge (3)
	(3)  edge (6)
	(6)  edge (7)
	(7)  edge (8)
	(8)  edge (9);
	
  \node (14) at (-3.5,0) {$\mathcal{G}_{n_k+2}$};
  \node (10) at (-2,0) {$\bullet$};
  \node (11) at (-1,0) {$\cdots$};
  \node (13) at  (0,0) {$\circ$};
  \node (12) at  (0,1) {$\circ$};
  \node (15) at  (1,.5) {$\bullet$}; 
  \node (16) at  (1,0) {$\bullet$}; 
  \node (17) at  (2,0) {$\bullet$} node [below=4pt] at (2,0) {$S_{n_k+2}$};
  \node (18) at  (3,0) {$\cdots$};
  \node (19) at  (4,0) {$\bullet$};

  \path[-] 
	(10)  edge (11)
	(12)  edge (15)
	(15)  edge (17)
	(11)  edge (13)
	(13)  edge (16)
	(16)  edge (17)
	(17)  edge (18)
	(18)  edge (19);
	
  \node (24) at (-3.5,-1) {$\mathcal{G}_{n_{k+1}}$};
  \node (20) at (-2,-1) {$\bullet$};
  \node (21) at (-1,-1) {$\cdots$};
  \node (23) at  (0,-1) {$\circ$};
  \node (26) at  (1,-1) {$\bullet$};
  \node (27) at  (2,-1) {$\bullet$};
  \node (28) at  (3,-1) {$\cdots$};
  \node (29) at  (4,-1) {$\bullet$} node [below=4pt] at (4,-1) {$S_{n_{k+1}}$};
  \node (30) at  (5,-1) {$\bullet$} node [below=4pt] at (5,-1) {$C_{n_{k+1}}$};
  \node (31) at  (6,-1) {$\cdots$};
  \node (32) at  (7,-1) {$\circ$};

\tikzstyle{every loop}=   [-, shorten >=.5pt]

  \path[-] 
	(20)  edge  (21)
	(21)  edge  (23)
	(23)  edge (26)
	(26)  edge (27)
	(27)  edge (28)
	(28)  edge (29)
	(29)  edge (30)
	(30)  edge (31)
	(31)  edge (32);
\end{tikzpicture}
\caption{The evolution of $\G_{n_k}$ along the path (I) $\to$ (II) $\to \cdots \to $ (II)  $\to$ (I). 
The vertex $\circ$ represents either $S_{n_k}$ or the extensions of $S_{n_k}$. }\label{cycle1}
\end{figure}






Compare with Sturmian words (see Figure~\ref{Rauzygraph}): there are infinitely many $n$'s such that 
the Rauzy graph has disjoint two cycles starting from a common bi-special word (see e.g. \cite{Ab}).
It corresponds to the factor graph $\G_n$ belongs to Case (I).  

\section{Quasi-Sturmian colorings of bounded type}\label{Sec_bounded}

In this section, we investigate a necessary and sufficient condition for a quotient graph to be a quotient graph of a quasi-Sturmian coloring of bounded type.

Let $x$ be a vertex of the quotient graph $X$.
For the two lifts $\tilde{x}$ and $\tilde{x}'$ of $x$, $[\B_n(\tilde{x})]=[\B_n(\tilde{x}')]$ for all $n$.
Then, $\tau(\tilde{x})=\tau(\tilde{x}')$.
By abuse of notation, define $[\B_n(x)]$ as a class $[\B_n(\tilde{x})]$.
Define the maximal type $\tau(x)$ of $x$ as $\tau(\tilde{x})$.

Recall the examples in Section~\ref{sec_Quotient graphs of quasi-Sturmian colorings}.
Let $\cX=(X,i)$ be the quotient graph for each of them.
We obtain a periodic edge-indexed subgraph $X'$ of $X$ by removing a finite subgraph $G$  {in Proposition \ref{Prop:TheQuotientGraphOfBoundedType}}.
Then, a lift of $(X',i|_{EX'})$ can be extended to a periodic coloring of a tree.
It is natural to guess that the property holds for every quasi-Sturmian coloring.

From now on, let $(\cT,\phi)$ be a quasi-Sturmian coloring of bounded type.
By Proposition~\ref{Prop:TheQuotientGraphOfBoundedType}, the quotient graph $X$ of $(\cT,\phi)$ is the graph in Figure~\ref{Figure:TheQuotientGraphOfBoundedType}.
Let $\widetilde{G}$ be the union of lifts of $G$.
A connected component of $\cT-\widetilde{G}$ is a lift of $(X-G,i|_{E(X-G)})$.
Thus, all connected components of $\cT-\widetilde{G}$ are equivalent to each other.
Let $Y$ be a connected component of $\cT-\widetilde{G}$.

\begin{lem}\label{Lem:OnY,UniquelyExtended}
If $u, v$ are vertices of $Y$ with $[\B_{\maxNN}(u)]=[\B_{\maxNN}(v)]$, where $\maxNN$ is as in (\ref{Eq:N_1}), 
then we have $[\B_{\maxNN+1}(u)]=[\B_{\maxNN+1}(v)]$. 
\end{lem}

\begin{proof}
It suffices to consider the case of $[\B_{\maxNN}(u)]= S_{\maxNN}$.
Every vertex of maximal type $\maxNN$ is the center of either $A_{\maxNN+1}$ or $B_{\maxNN+1}$, say  $A_{\maxNN+1}$. 
Since vertices of $X - G$ are of maximal type bigger than $\maxNN$,
if $u$ is a vertex of $Y$ and $[\B_{\maxNN}(u)]= S_{\maxNN}$, then $[\B_{\maxNN+1}(u)]= B_{\maxNN+1}$.
\end{proof}

We define an edge-indexed graph  {$\cZ = (Z, i_Z)$} as follows :  {the vertices of $Z$ are of the form} $[\B_{\maxNN}(u)]$ for  {a vertex} $u$ in $Y$ or $X - G$, and  {any two vertices $D, E$ of $Z$ are adjacent if $D$ and $E$ are weakly adjacent.}
The index  {$i_Z(D,E)$ is the number of $E$ which are adjacent to $D$.}
The indices are well-defined by Lemma~\ref{Lem:OnY,UniquelyExtended}.
Since any vertex in $X - G$ is adjacent to at most two vertices besides itself, the graph $Z$ is a line segment or a cycle.

\begin{lem}\label{lem_periodic_ext}
A restriction of $\phi$ on any connected component of $\cT-\widetilde{G}$ has a periodic extension to $\cT$.
\end{lem}

\begin{proof}
Let $u$ be the vertex of $Y$. 
Define a coloring $\psi_k$ on $\B_{k}(u)$ with  {the alphabet $VZ = \{ [ \B_{\maxNN}(v)] \, | \, v \in Y \}$  recursively}: 
Put $\psi_0 (u) = [\B_{\maxNN}(u)] \in VZ$. 
Define $\psi_{k+1} (v) = \psi_k (v)$ for $v \in \B_{k}(u)$.
Choose $w \in V\cT$ with $\mathbf d (u,w) = k$ and let $w_\alpha$ ($\alpha = 0, \dots d-1$)  be  the neighboring vertices of $w$ with $\mathbf d (u,w_\alpha) = k+1$ for $\alpha \ge 1$ and $\mathbf d (u,w_0) = k-1$.
 {We define $\psi_{k+1} (w_\alpha)$ for $\alpha \ge 1$ in the following ways.}

If $w \notin Y$, then $w_\alpha \notin Y$ for all $\alpha \ge 1$.
 {Let $D_0 = \psi_{k}(w_0)$ and $D_j$ be satisfying  
$i_Z( \psi_{k}(w), D_j) > 0$ with $j = 0, 1, 2$ or $j = 0,1$.}
We assign $\psi_{k+1} (w_\alpha)$  { for $1 \le \alpha \le d-1$ as 
$$ 
\psi_{k+1}(w_\alpha) =  D_\ell \ \text{ for } \ 
\sum_{j=0}^{\ell-1} i_Z(\psi_{k}(w), D_j ) \le \alpha \le \sum_{j=0}^{\ell} i_Z(\psi_{k}(w), D_j ) -1.
$$
Then we have}
\begin{equation}\label{index_formula}
i_Z(\psi_{k+1}(w), D) = \# \{ 0 \le \alpha \le d \, |  \, \psi_{k+1}(w_\alpha) = D \}
\end{equation}
for each $D \in VZ$.

If $w \in Y$, then
we put $\psi_{k+1}(w_\alpha) =  [\B_{\maxNN}(w_\alpha)]$ for all $\alpha \ge 1$.
 {Using the fact that $Y$ is an infinite subgraph of $T$, 
Lemma~\ref{Lem:OnY,UniquelyExtended} implies that there exists a vertex $v$ such that 
$\B_{\maxNN+1}(v) \subset Y$ and $[\B_{\maxNN+1}(v) ] = [\B_{\maxNN+1}(w)]$, 
thus $\psi_{k+1}(w_\alpha) = [\B_{\maxNN}(w_\alpha)] \in VZ$}
 and \eqref{index_formula} is satisfied.

Since $\psi_{k+\ell} |_{\B_{k}(u)} = \psi_k$ for $\ell \ge 1$, the coloring $\psi = \lim_{k \to \infty} \psi_k$ on $\cT$ with alphabet $VZ$ exists. By \eqref{index_formula}, we deduce that  {$\cZ$} is the quotient graph of $\psi$.
Since $\psi (u) = [\B_{\maxNN}(u)]$ on $Y$, by the coloring which gives the color of the center of $\psi(u)$, 
we complete the proof.
\end{proof}

\begin{thm}\label{Thm:AnIffConditionOfTheQuotientGraphOfBoundedType}
Let $\cX=(X,i)$ be the quotient graph of a coloring $(\cT,\phi)$.
The following statements are equivalent.
\begin{enumerate}
\item The coloring $\phi$ is a quasi-Sturmian coloring of bounded type.
\item There is a finite connected subgraph $G$ of $X$ such that $X-G$ is a connected infinite ray and any connected component of $\cT-\widetilde{G}$ has a periodic extension to $\cT$ where $\widetilde{G}$ is the union of lifts of $G$.
\end{enumerate}
\end{thm}


\begin{proof}
By Lemma~\ref{Lem:OnY,UniquelyExtended} and Lemma~\ref{lem_periodic_ext}, (1) implies (2).
Now we assume (2) holds.
Let $\mathcal{A}$ be the alphabet of $\phi$.
Let $\tilde{x}$ be a lift of $x\in VX$.
Define a new coloring $\psi$ with an alphabet $\mathcal{A}\sqcup VG$ as
\begin{equation*}
\psi(v)=
\begin{cases} x      & \text{ if }  v=\tilde{x} \text{ for some } x\in VG, \\
                                \phi(v) &\text{ otherwise}.
\end{cases}
\end{equation*}
Denote by $[\B_n(u)]_\psi$ a $\psi$-colored $n$-ball. 
As ever $[\B_n(u)]$ means a $\phi$-colored $n$-ball.
A map $\mathbb{B}_\psi(n)\rightarrow \mathbb{B}_\phi(n)$ which defined by $[\B_n(x)]_\psi\mapsto[\B_n(x)]$ is surjective. 
It implies $b_\phi(n) \le b_\psi(n)$.
Since $X$ is not a finite graph, $b_\phi(n)$ is strictly increasing.
Thus, it is enough to show that $b_\psi$ is linear.

 {Let us denote by $\metricd(x,G)=\min\{\metricd(x,g):g\in VG\}$.}
Fix a positive integer $n$.
If $x$ be a vertex such that $\metricd(x,G)\le n$, 
then $[\B_n(x)]_\psi\not=[\B_n(y)]_\psi$ for any other $y\in VX$.
If $x$ be a vertex such that $\metricd(x,G) > n+1$,
then $[\B_{n+1}(x)]_\psi=[\B_{n+1}(x)]$.
Thus, $[\B_{n}(x)]$ has the unique extension to a colored $(n+1)$-ball.
Since $X$ is not finite, $\psi$ has at least one special $n$-ball for each $n$.
Thus,  {for $x$ such that $\metricd(x,G)=n+1$, $[\B_n(x)]$ is the unique special $n$-ball and it has exactly two extensions to colored $(n+1)$-balls.}
It means that $b_\psi(n)=n+|\mathcal{A}|+|VG|$ for all $n$.
\end{proof}


\section{Recurrence functions of colorings of trees}\label{Sec_recurrence}

In this section, we will extend the notion of recurrence functions $R(n), R''(n)$ for words to colorings of trees. 
We will show that the quasi-Sturmian colorings of trees satisfy a certain inequality between $R''(n)$ and $b(n)$. We also explain that the existence of $R(n)$ is related to unboundedness of the quasi-Sturmian colorings of trees.

Let us briefly recall recurrence functions of words  {(see Section 10.9 in \cite{AS} for definitions and details)}. 
Recurrence functions are important objects related to symbolic dynamics. 
Let $\Sigma$ be a finite alphabet. Let $\Sigma^*$ be the set of finite words over $\Sigma$ and $\Sigma^\mathbb{N}$ be the set of infinite words over $\Sigma$.
For $\mathbf{u} \in \Sigma^* \cup \Sigma^\mathbb{N}$, we denote by $F_n (\mathbf{u})$ the set of factors of length $n$ of $\mathbf{u}$.

A recurrence function $R_{\mathbf{u}}(n)$ is defined as the smallest integer $m \ge 1$ such that every factor of length $m$ contains all factors of length $n$.
%
%
%
%
It is known that such  {an integer} $R_{\mathbf{u}}(n)$ exists for all $n$ if and only if the word is \emph{uniformly recurrent}, i.e. any subword of the word infinitely occurs with bounded gaps. 
Another recurrence function $R''_\mathbf{u}(n)$ is defined as
$$
R''_\mathbf{u}(n)=\min\{m\in \mathbb{N}~|~F_n(\mathbf{u})=F_n(\omega)\text{ for some }\omega \in F_m (\mathbf{u}) \},
$$
i.e. it is the length of the smallest factor of $\mathbf{u}$ that contains all factors of length $n$ of $\mathbf{u}$. From the definition, the following fact immediately holds.

\begin{remark} \label{rmk6.1_qs}
For all $n \ge 0$, $R''_\mathbf{u}(n) \ge p_\mathbf{u}(n)+n-1$ for any word $\mathbf{u}$.
\end{remark}

Recall that a word $\mathbf{u}$  {is said to have} \emph{grouped factors} if, for all $n \ge 0$, it satisfies $ R''_\mathbf{u}(n)=p_\mathbf{u}(n)+n-1$. 
If there is $n_0$ such that the equality holds for all $ n \ge n_0 $, we say that $\mathbf{u}$ has \emph{ultimately grouped factors}.
Cassaign suggested some conditions that guarantee the equality.

\begin{thmnn}[\cite{Ca}]
A word $\mathbf{u}$ is Sturmian if and only if $R''_\mathbf{u}(n) =2n$ for every $n \ge 0$.
A uniformly recurrent word on a binary alphabet has ultimately grouped factors if and only if it is periodic or quasi-Sturmian.
\end{thmnn}

We want an analogous statement for quasi-Sturmian colorings of trees.
Let $(\cT,\phi)$ be a quasi-Sturmian coloring of a tree and $\cX=(X,i)$ be the quotient graph of $(\cT,\phi)$. 
We define $R_\phi(n)$ as the smallest radius $m$ such that every colored $n$-ball of $\phi$ occurs in $[\B_m(x)]$ for all $x\in V\cT$.
We define $R''_\phi(n)$ as the smallest radius $m$ such that every colored $n$-ball of $\phi$ occurs in $[\B_m(x)]$ for some $x\in V\cT$.


\begin{defn} A coloring of a tree $(\cT, \phi)$ is said to be \emph{recurrent} if, for any compact subtree $\cT'$, every colored ball appears in $\cT-\cT'$. 
A coloring of a tree is said to be \emph{uniformly recurrent} if $R_\phi(n)<\infty$ for all $n$.
\end{defn}


\begin{prop}\label{unbdd_unif}
Let $(\cT,\phi)$ be a quasi-Sturmian coloring of a tree. The following conditions are equivalent.
\begin{enumerate}

\item $(\cT,\phi)$ is of unbounded type.
\item $(\cT,\phi)$ is uniformly recurrent.
\item For any colored ball, it appears in $\cT- \pi^{-1}(S)$ for any finite set $S \subset X$.
\end{enumerate}

\begin{proof}

%
%
%
%

(1) implies (2) : 
Suppose $(\cT,\phi)$ is of unbounded type. Let $n \ge \initN$.
For arbitrary $v\in V\cT$, let $D=[\B_n(v)]$.
Consider $E=[\B_{n}(w)]$ which is distinct with $D$.
Since $\Lambda_w$ is infinite, there exists the minimal number $m=m_E \ge n$ in $\Lambda_w$.
Note that $m$ depends only on $E$ and not on $w$.

Let $F^1=[\B_{m}(v)]$. It is not $S_m$.
Let $[F^1F^2\cdots F^{l}S_{m}]$ be the shortest path from $F^1$ to $S_{m}$ in $\G_{m}$.
For arbitrary colored $m$-balls $F$ and $F'$, if $F\not=S_m$, then $F$ has the unique extension. Thus, if $F$ is weakly adjacent to $F'$, then $F$ is strongly adjacent to $F'$.
Therefore, there is a path $[v-v_2-v_3-\cdots-v_{l}-w']$ in $\cT$ such that $[\B_{m}(v_i)]=F^i$, $i=2,\cdots, l$, and $[\B_{m}(w')]=S_{m}$.

Since $S_m$ occurs in $[\B_{m+l}(v)]$, $E$ occurs in $[\B_{n+l}(v)]$.
Since $l\le |V\G_{m}|=m+c$, $E$ occurs in $[\B_{n+m+c}(v)]$.
Define $m_D=0$.
Every colored $n$-ball occurs in $[\B_{n+M+c}(v)]$ where $M=\max\{m_E:E\in\mathbb{B}_\phi(n)\}$.
Thus, $R_\phi(n)\le n+M+c$.



(2) implies (3) : Suppose that $R_\phi(n)$ exists for all $n$. 
Since the quotient graph $X$ is infinite, for any finite $S\subset X$, there is $x$ such that $\B_{R_\phi(n)}(x)\subset \cT-\pi^{-1}(S)$.

(3) implies (1) : Assume that ($\cT$,$\phi$) is of bounded type.
Let $v$ be a vertex of maximal type $\maxNN$.
By Proposition~\ref{Prop:TheQuotientGraphOfBoundedType}, all vertices in $X - G$ is of maximal type larger than $\maxNN$.
Therefore, $[\B_{\maxNN+1}(v)]$ does not appear in $\cT - \pi^{-1}(G)$.
\end{proof}
\end{prop}

Recall that we denote by  {$\cZ$} the quotient graph of $\cT-\widetilde{G}$ with respect to the coloring $\phi$ and denote by $\radiusr(x,G)$
$$
\mathrm{\radiusr}(x,G)=\max\{\metricd(x,y):y\in VG\}.
$$

\begin{prop}\label{unbdd_ineq}
Let $(\cT,\phi)$ be a quasi-Sturmian coloring.
\begin{enumerate}
\item Let $\phi$ be of unbounded type. As in Proposition ~\ref{graph_develop}, the factor graph $\G_n$ is of Case (I) on $n=n_k$.
Then, we have 
$$
R_\phi''(n) = n+ \Bigl\lfloor\frac{b_\phi(n_k)}{2}\Bigr\rfloor
								 \quad \mathrm{for} \; \; n_{k-1}<n \le n_k.
$$ 

\item Let $\phi$ be of bounded type. Let $\xNone$ be the vertex of $X$ which is of maximal type $\maxNN$.

\begin{enumerate}
\item If $Z$ is acyclic, then we have 
$$
R_\phi''(n) = n+ \Bigl\lfloor\frac{1}{2}(b_\phi(n_k)-|G|+\radiusr(\xNone,G)+1)\Bigr\rfloor
				\quad \mathrm{for} \; \;n_{k-1}<n \le n_k.
$$ 

\item If $Z$ is cyclic, then we have 
$$
R_\phi''(n) = n+\Bigl\lfloor\frac{1}{2}(b_\phi(n)-|G|+\radiusr(\xNone,G)+1)\Bigr\rfloor 
				\quad \mathrm{for} \; \mathrm{all} \; \; n \ge \maxNN.
$$ 
\end{enumerate}
\end{enumerate}
\end{prop}

\begin{proof}
(1) 
In the case of a quasi-Sturmian coloring of unbounded type, the evolution of the factor graph follows Proposition~\ref{graph_develop}. Then, we can choose an infinite sequence $\{n_k\}$ such that $\G_{n_k}$ is in Case (I) of Proposition~\ref{graph_develop}. 
For any colored $n_k$-balls $D$ and $E$, assume that $D$ is weakly adjacent to $E$. 
If $D$ is not a special ball, then $D$ has a unique extension. Hence, $D$ is strongly adjacent to $E$. If $D=S_{n_k}$, then assume that $D$ is weakly adjacent to $E$ and $F$. Since one of $E$ and $F$ is $C_{n_k}$, say $E$, $D$ is strongly adjacent to $E$ by Lemma ~\ref{lem2.3_qs}. Hence, there exist vertices $v$,$u$ and $w$ in $\cT$ with $\metricd(v,u)=\metricd(v,w)=1$ such that $D=[\B_{n_k}(v)]$, $E=[\B_{n_k}(u)]$, and $F=[\B_{n_k}(w)]$ in $\cT$. Therefore, We can take a path with length $b_\phi(n_k)-1$ consisting of centers of all the colored $n_k$-balls in $\cT$. 

Thus, we have 
$$
R_\phi''(n_k)\le n_k + \Bigl\lfloor\frac{b_\phi(n_k)}{2}\Bigr\rfloor.
$$
Let $D_{n_k}$, $E_{n_k}$ be the colored $n_k$-balls which are the end points of the graph $\G_{n_k}$. The distance between  $D_{n_k}$ and $E_{n_k}$ in $\G_{n_k}$ is $b_\phi(n_k)-1$, so  {for any vertices $z,z'$ such that $[\B_{n_k}(z)]=D_{n_k}$ and $[\B_{n_k}(z')]=E_{n_k}$, $\metricd(z,z') \ge b_\phi(n_k)-1$.} Thus, it implies 


$$
R_\phi''(n_k) \ge n_k + \Bigl\lfloor\frac{b_\phi(n_k)}{2}\Bigr\rfloor.
$$

Hence, we have 
$$
R_\phi''(n_k) = n_k + \Bigl\lfloor\frac{b_\phi(n_k)}{2}\Bigr\rfloor.
$$

Now, let us consider the case $n_{k-1}<n<n_k$. 
 {Then, $\G_n$ is of Case (II) or Case (III). Let us define two colored $n$-balls $D_n$ and $E_n$. If $\G_n$ is of Case (II), then $D_n$ and $E_n$ is defined as the colored $n$-balls which are the end points of two paths starting from $S_n$ to $A_n$ and $B_n$ in $\G_n$, respectively. If $\G_n$ is of Case (III), then $D_n$ and $E_n$ is defined as the end points of $\G_n$, respectively. Now, let us compute the distance between $D_n$ and $E_n$.}

Let $D$, $E$ be colored $n$-balls. If $D \ne S_n$ is weakly adjacent to $E$, then $D$ is strongly adjacent to $E$. However, if $D=S_n$, then $D$ is stronly adjacent to either $A_n$ and $C_n$ or $B_n$ and $C_n$. If $\G_n$ is of Case (II) for all $n_{k-1}<n<n_k$, then $\metricd(D_n,E_n) \ge b_\phi(n_k)-1$. Otherwise, $\G_n$ is of Case (III) only for $n=n_k-1$ and it is of Case (II) for $n\ne n_k-1$. Then, $\metricd(D_n,E_n) \ge b_\phi(n_k-1)-1+1$. Hence, $\metricd(D_n,E_n) \ge b_\phi(n_k)-1$, which implies 
$$
R_\phi''(n) \ge n + \Bigl\lfloor\frac{b_\phi(n_k)}{2}\Bigr\rfloor
$$ for $n_{k-1}<n<n_k$.
On the other hand, since each $n$-ball is the restriction of an ${n_k}$-ball and there exists the path with length $b_\phi(n_k)-1$ consisting of centers of all the colored $n_k$-balls in $\cT$, we have
$$ 
R_\phi''(n) \le n_k-(n_k-n) + \Bigl\lfloor\frac{b_\phi(n_k)}{2}\Bigr\rfloor
				= n+ \Bigl\lfloor\frac{b_\phi(n_k)}{2}\Bigr\rfloor
$$ for $n_{k-1}<n<n_k$. 
Thus, we have 
$$
R_\phi''(n) = n + \Bigl\lfloor\frac{b_\phi(n_k)}{2}\Bigr\rfloor.
$$ 
for $n_{k-1}<n \le n_k$.

(2)-(a) 
Let $Z$ be acyclic. Assume that $n\ge \maxNN$. 
The evolution of the factor graph $\G_n$ follows Proposition~\ref{graph_develop}.
Hence, we can consider the same argument with (1). The difference between (1) and (2)-(a) is the existence of the compact part $G$.  {Now, we can take a finite graph $G'$ in $\G_{n_k}$ isomorphic to $G$.} Since every vertex in $\G_{n_k}-G'$ has at most degree 2, the maximal distance between any two vertices in $\G_{n_k}$ is $b_\phi(n_k)-|G|+\radiusr(\xNone,G)$.
 
Thus$$
R_\phi''(n_k) \ge n_k + \Bigl\lfloor\frac{1}{2}(b_\phi(n_k)-|G|+\radiusr(\xNone,G)+1)\Bigr\rfloor.
$$

 {Now, we can choose a path $P$ in $\cT$ isomorphic to $\G_{n_k}$, i.e. there exsists a bijection $f$ : $VP \rightarrow V\G_{n_k}$ such that two vertices $u$ and $v$ are adjacent in $P$ if and only if $f(u)$ (respectively $f(v)$) is weakly adjacent to $f(v)$ (respectively $f(u)$) in $\G_{n_k}$. This is because weak adjacency implies strong adjacency by the same argument with (1).} Thus, $$
R_\phi''(n_k) \le n_k + \Bigl\lfloor\frac{1}{2}(b_\phi(n_k)-|G|+\radiusr(\xNone,G)+1)\Bigr\rfloor.
$$

Hence, $$
R_\phi''(n_k) = n_k + \Bigl\lfloor\frac{1}{2}(b_\phi(n_k)-|G|+\radiusr(\xNone,G)+1)\Bigr\rfloor.
$$
Now, let us consider the case $n_{k-1}<n<n_k$. Then, we can compute $R_\phi''(n)$ by the same aregument with (1). The difference between (1) and (2)-(a) is also the existence of the compact part $G$.
Thus, 
$$
R''_\phi(n)=n+\Bigl\lfloor\frac{1}{2}(b_\phi(n_k)-|G|+\radiusr(\xNone,G)+1)\Bigr\rfloor
$$
for $n_{k-1}<n<n_k$. 
Hence, we have 
$$
R''_\phi(n)=n+\Bigl\lfloor\frac{1}{2}(b_\phi(n_k)-|G|+\radiusr(\xNone,G)+1)\Bigr\rfloor
$$ for $n_{k-1}<n \le n_k$.

(2)-(b) 
Let $Z$ be cyclic. Assume that $n \ge \maxNN$. The special $n$-ball is of degree 3 in $\G_n$. Moreover, $S_n$ is uniquely of degree 3 not overlapping the finite graph $G$ in Proposition ~\ref{Thm:TheQuotientGraphOfQS}. 
The cycle connecting to $S_n$ in $\G_n$ is also unique except for cycles in $G$.  {Since weak adjacency implies strong adjacency, we can take a finite graph $G''$ in $\cT$ isomorphic to $G$, i.e. there exsists a bijection $g$ : $VG'' \rightarrow VG$ such that two vertices $u$ and $v$ are adjacent in $G''$ if and only if two vertices $g(u)$ and $g(v)$ are adjacent in $G$.}
To contain all of the colored $n$-balls, it is sufficient for an $R_\phi ''(n)$-ball to contain the finite graph $G''$ and the path $[A_n...C_nS_nB_n...[\B_n(\tildexNone)]]$, where a vertex $\tildexNone$ is a lifting of $\xNone$ to $\cT$. Since the length of the path is $b_\phi(n)-|G|$, we have
$$
R''_\phi(n)=n+\Bigl\lfloor\frac{1}{2}(b_\phi(n)-|G|+\radiusr(\xNone,G)+1)\Bigr\rfloor
$$
when the equality holds if and only if $G$ is linear.
\end{proof}


 {We note that the converse of the proposition does not hold.
Consider a sequence of words 
$$X_k = \begin{cases} a L_k a L_k b L_k a , &\text{ if  $k$ is odd},\\
b L_k a L_k b L_k b, &\text{ if  $k$ is even}, \end{cases}$$ where $L_k$ is given by 
$L_1 = \varepsilon$, the empty word and $L_{k+1} = L_k a L_k$ for odd $k$, $L_{k+1} = L_k b L_k$ for even $k$ recursively.
Then $L_k$ is a palindrome and  we get
$$ X_1 = a a b a, \qquad X_2 = baaabab , \qquad X_3 = a aba a aba b aba a, \qquad \dots $$
Since $X_k$ is a factor of $X_{k+1}$, we have a coloring $\phi$ of the 2-regular tree by the limit of $X_k$. 
Let $n_k = | L_k a_k L_k | =2^k-1$. 
Then we can check  that
for $n_{k-1} < n \le n_k$, we have
$$
R_\phi''(n) - n =  \Bigl\lfloor\frac{ | X_k| }{2}\Bigr\rfloor
$$
and $$
b_\phi(n_k)  = | X_k|. 
$$ 
Thus, we have
$$
R_\phi''(n) = n+ \Bigl\lfloor\frac{b_\phi(n_k)}{2}\Bigr\rfloor
								 \quad \mathrm{for} \; \; n_{k-1}<n \le n_k.
$$ 
}


\end{document}